\documentclass[12pt,twoside,reqno]{amsart}
\usepackage[colorlinks=true,citecolor=blue]{hyperref}
\usepackage{mathptmx, amsmath, amssymb, amsfonts, amsthm, mathptmx, enumerate, color,mathrsfs}
\usepackage{graphicx}

\usepackage{multirow}
\usepackage{epstopdf}
\usepackage{multicol}
\usepackage{algorithm}
\usepackage{algorithmic}
\usepackage{epstopdf}
\usepackage{psm-math}

\newtheorem{theorem}{Theorem}[section]

\newtheorem{corollary}[theorem]{Corollary}

\theoremstyle{definition}
\newtheorem{definition}[theorem]{Definition}

\newtheorem{example}[theorem]{Example}

\newtheorem{remark}[theorem]{Remark}
\numberwithin{equation}{section}

\begin{document}
\setcounter{page}{1}

\vspace*{2.0cm}
\title[On Semidefinite Representations of Second-order Conic Optimization Problems]
{On Semidefinite Representations of Second-order Conic Optimization Problems}
\author[Sampourmahani, Mohammadisiahroudi, Terlaky]{ Pouya Sampourmahani$^{*}$, Mohammadhossein Mohammadisiahroudi, Tamás Terlaky}
\maketitle
\vspace*{-0.6cm}

\begin{center}
{\footnotesize
Department of Industrial and Systems Engineering, Lehigh University, Bethlehem, PA, USA, 18015\\
}\end{center}

\vskip 4mm {\footnotesize \noindent {\bf Abstract.}
Second-order conic optimization (SOCO) can be considered as a special case of semidefinite optimization (SDO). In the literature it has been advised that a SOCO problem can be embedded in an SDO problem using the arrow-head matrix transformation. However, a primal-dual solution pair cannot be mapped simultaneously using the arrow-head transformation as we might lose complementarity and duality in some cases. To address this issue, we investigate the relationship between SOCO problems, and their SDO counterpart. Through derivation of standard semidefinite representations of SOCO problems, we introduce admissible mappings. We show that the proposed mappings preserve both feasibility and optimality. Further, we discuss how the optimal partition of a SOCO problem maps to the optimal partition of its SDO counterpart.

 \noindent {\bf Keywords.}
 Second-order conic Optimization; Semidefinite Optimization; Semidefinite Representation; Mapping; Optimal Partition. 

 \noindent {\bf 2020 Mathematics Subject Classification.}
90C25, 90C22, 90C99. }

\renewcommand{\thefootnote}{}
\footnotetext{ $^*$Corresponding author.
\par
E-mail address: pos220@lehigh.edu (P. Sampourmahani), mom219@lehigh.edu (M. Mohammadisiahroudi), terlaky@lehigh.edu (T. Terlaky).
\par
Received xx, x, xxxx; Accepted xx, x, xxxx.

\rightline {\tiny   \copyright  2023 Communications in Optimization Theory}}

\section{Introduction}
In the hierarchy of convex optimization problems, second-order conic optimization (SOCO) problems can be seen as a special case of semidefinite optimization (SDO) problems. SOCO problems minimize a linear function over the intersection of an affine space with the Cartesian product of second-order cones, also known as Lorentz cones. An SDO problem consists of minimizing a linear objective function over the intersection of the cone of positive semidefinite matrices with an affine space. SDO encompasses other subclasses of conic optimization problems namely linear optimization (LO), and SOCO, in the hierarchy. This means that each one can be represented as a special case of  SDO \cite{alizadeh2003second}.

In this paper, we focus on the relationship of SOCO and SDO. We investigate their relationship in order to gain theoretical insight and realize how these problems get mapped to each other. Only a few papers \cite{sim2005note,zhou2016further} were devoted to study this relationship from a theoretical point of view. Sim and Zhao \cite{sim2005note}, in particular, studied the relationship between a SOCO problem and its counterpart SDO problem. They provided a mapping based on the direct correspondence between the dual problems of SOCO and SDO. Their SDO representation is defined on the product of some cones of positive semidefinite matrices, which is a special case of standard SDO and needs further analysis. In this paper, we extend their analysis by considering the actual standard case which returns an SDO representation through a large positive semidefinite cone.

Furthermore, we propose a framework that allows full description of the point-to-set map from SOCO to its SDO counterpart. Then, we analyze how the optimal partition of a SOCO problem is mapped to that of SDO, and vice versa. This is important in understanding the relationship between these two problems as we are mapping between an index-based partition and a subspace-based partition.

Throughout this paper, the following notation is used. The Lorentz cone of dimension $n_i$ is denoted by $\Lcal^{n_i}$, and $\Rmbb^n$ denotes the $n$-dimensional Euclidean space. Superscripts are used to represent cone-related information, and subscripts are used for matrix and vector entries. For a given matrix $A$, $A_{ij}$ represents the $(i,j)$-th entry, while $A_{i}$ denotes the $i$-th matrix. The notation $(.;.;...;.)$ denotes the concatenation of the column vectors.  The set of all $p \times q$ matrices with real entries is denoted by $\Rmbb^{p \times q}$. For a symmetric matrix $X$, $X \succeq 0$ ($X \succ 0$) means $X$ is positive semidefinite (positive definite). Furthermore, the trace operator is denoted by $\trace(.)$. The remaining notations will be introduced at appropriate places.

This paper is structured as follows. Section \ref{sec:prelim} reviews the preliminaries required for this paper. Section \ref{sec:3} studies the relationship between SOCO and SDO relying on the correspondence of dual problems. Section \ref{sec:4} takes the other direction and proposes mappings focusing on correspondence of primal problems as the starting point. Section \ref{sec:5} analyzes how the optimal partition of SOCO maps to that of it's SDO counterpart. Section \ref{sec:6} concludes the paper, summarizing our results.

\section{Preliminaries} \label{sec:prelim}
Let $\Lcal^{n_i}$ denote the Lorentz cone of dimension $n_i$, and $\Lcal^{\nbar} = \Lcal^{n_1} \times \Lcal^{n_2} \times ... \times \Lcal^{n_r}$, where $\nbar = \sum_{i=1}^r n_i$. Then, the primal and dual SOCO problems are defined as follows
\begin{equation} \label{mo:gensocop} \tag{${\Pcal}_{SOCO}$}
    \begin{aligned}
    z_{\Pcal_{SOCO}}^* := \min \ & (c^{1})^T \xbar^1 + ... + (c^{r})^T \xbar^r \\
    \text{  s.t.} \ & A^1\xbar^1 + ... + A^r \xbar^r  = b, \\ 
    \ & \xbar^i \in \Lcal^{n_i}, && \text{ for } i =1,...,r,  
    \end{aligned}
\end{equation}

\begin{equation} \label{mo:gensocod} \tag{${\Dcal}_{SOCO}$} 
    \begin{aligned}
    z_{\Dcal_{SOCO}}^* := \max \ & b^T \ybar  \\
    \text{  s.t.} \ & (A^i)^T \ybar + \sbar^i = c^i, && \text{ for } i =1,...,r, \\
    \ & \sbar^i \in \Lcal^{n_i}, && \text{ for } i =1,...,r, 
    \end{aligned}
\end{equation}
where $c^i \in \Rmbb^{n_i}$, $A^i \in \Rmbb^{m\times n_i}$, $b \in \Rmbb^m$. We define the feasible sets of the primal-dual problems as follows,
\begin{align*}
    \Fcal_{{\Pcal}_{SOCO}} &= \{ (\xbar^1;\xbar^2;...;\xbar^r) \in \Lcal^{\nbar}: A^1\xbar^1 + ... + A^r \xbar^r  = b\},\\
    \Fcal_{{\Dcal}_{SOCO}} &= \{ (\ybar;\sbar^1;\sbar^2;...;\sbar^r) \in \mathbb{R}^m\times \Lcal^{\nbar}: (A^i)^T \ybar + \sbar^i = c^i \text{ for } i =1,...,r\},
\end{align*}
and the sets of optimal solutions as
\begin{align*}
    {\Pcal}_{SOCO}^* &= \{ \xbar^* = (\xbar^1;\xbar^2;...;\xbar^r) \in \Fcal_{{\Pcal_{SOCO}}}: c^T\xbar^*=z_{{\Pcal}_{SOCO}}^*\},\\
    {\Dcal}_{SOCO}^* &= \{ (\ybar^*, \sbar^*) = (\ybar;\sbar^1;\sbar^2;...;\sbar^r) \in \Fcal_{{\Dcal_{SOCO}}}: b^T\ybar^*=z_{{\Dcal}_{SOCO}}^*\},
\end{align*}
respectively. An optimal solution of SOCO, if there exists any, is denoted by $(\xbar^*;\ybar^*;\sbar^*)$. 

Let $\arw(\cdot)$ denote the arrow-head (Lorentz) transformation \cite{mohammad2019conic,wolkowicz2012handbook}, with the structure of
\begin{align*}
    \arw(\xbar^i) := 
    \begin{pmatrix}
    \xbar_1^i       & (\xbar_{2:n_i}^i)^T \\
    \xbar_{2:n_i}^i & \xbar_1^i I_{n_i -1}
    \end{pmatrix},
\end{align*}
where $(\xbar^i_{2:n_i})^T$ denotes the vector $(\xbar^i_2, ..., \xbar^i_{n_i})$. Then, the Jordan product is defined as
\begin{equation*}
    \xbar^i \circ \sbar^i = \arw(\xbar^i)\sbar^i = \arw(\sbar^i)\xbar^i = 
    \begin{pmatrix}
    (\xbar^i)^T \sbar^i\\
    \xbar^i_1 \sbar^i_{2:n_i} + \sbar^i_1 \xbar^i_{2:n_i}
    \end{pmatrix},
    \quad
    i = 1,...,r.
\end{equation*}
Any feasible solutions satisfying $\xbar \circ \sbar = 0$ is called complementary. Here, we have
\begin{equation*}
    \xbar \circ \sbar := (\xbar^1 \circ \sbar^1, \xbar^2 \circ \sbar^2, ..., \xbar^r \circ \sbar^r).
\end{equation*}
 Feasible solution are complementary if and only if they are optimal with zero duality gap.
\begin{definition}
An optimal solution $(\xbar^*; \ybar^*; \sbar^*)\in {\Pcal}_{SOCO}^* \times {\Dcal}_{SOCO}^*$ is called maximally complementary if $\xbar^* \in \ri({\Pcal}_{SOCO}^*)$ and $(\ybar^*; \sbar^*) \in \ri({\Dcal}_{SOCO}^*)$. Further, $(\xbar^*; \ybar^*; \sbar^*)$ is called strictly complementary if $\xbar^* + \sbar^* \in \interior(\Lcal^{\nbar})$.
\end{definition}

Next, we define the primal and dual SDO problems.

\begin{equation} \label{mo:gensdop} \tag{${\Pcal}_{SDO}$}
    \begin{aligned}
     z_{\Pcal_{SDO}}^* := \min  \ & \trace{(CX)} \\
     \text{ s.t.} \ & \trace{(A_i X)} = b_i && \text{ for all } i = 1,...,m, \\
     \ & \quad X \succeq 0, 
    \end{aligned}
\end{equation}
\begin{equation} \label{mo:gensdod} \tag{${\Dcal}_{SDO}$} 
    \begin{aligned}
    z_{\Pcal_{SDO}}^* := \max \ &  b^T y \\
    \text{ s.t.} \ & \sum_{i=1}^m y_i A_i + S = C, \\
    \ & S \succeq 0, 
    \end{aligned}
\end{equation}

where $X, S, C$, and $A_i \text{ for }i=1,...,m$ are $n \times n$ symmetric matrices, and $b,y \in \Rmbb^m$. We define the feasible sets of SDO problems as
\begin{align*}
    \Fcal_{{\Pcal}_{SDO}} &= \{ X \in \Smbb^n : \trace{(A_i X)} = b_i, i = 1,...,m, X \succeq 0\},\\
    \Fcal_{{\Dcal}_{SDO}} &= \{(y,S) \in \Rmbb^m \times \Smbb^n : \sum_{i=1}^m y_i A_i + S = C, S \succeq 0\},
\end{align*}
where $\Smbb^n$ denotes the set of $n\times n$ symmetric matrices.
The sets of optimal solutions for a pair of SDO problems are
\begin{align*}
    {\Pcal}_{SDO}^* &= \{ X^* \in \Fcal_{{\Pcal}_{SDO}} : \trace{(CX^*)} = z^*_{{\Pcal}_{SDO}}\},\\
    {\Dcal}_{SDO}^* &= \{(y^*,S^*) \in \Fcal_{{\Dcal}_{SDO}}: b^T y^* = z^*_{{\Dcal}_{SDO}}\}.
\end{align*} 
 A feasible and an optimal solution of SDO are denoted as $(X,y,S)$, and $(X^*,y^*,S^*)$, respectively. Any feasible solution $(X,y,S)$ satisfying $XS = 0$ is called complementary. Similar to SOCO, a feasible solution is optimal and yields zero duality gap if and only if it is complementary. 
\begin{definition}
A primal-dual optimal solution $(X^*,y^*,S^*) \in {\Pcal}_{SDO}^* \times {\Dcal}_{SDO}^*$ is called maximally complementary if $X^* \in \ri({\Pcal}_{SDO}^*)$ and $(y^*,S^*) \in \ri({\Dcal}_{SDO}^*)$. A maximally complementary optimal solution $(X^*,y^*,S^*)$ is called strictly complementary if $X^* + S^* \succ 0$.
\end{definition}

Next, we define the optimal partitions of SOCO and SDO. The notion of the optimal partition of LO can be extended to SOCO \cite{mohammad2019conic}. Even though a SOCO problem can be embedded in SDO, the optimal partition in SOCO may be more nuanced when it is defined and analyzed directly in the SOCO setting. In SOCO, the index set $\{1,...,r\}$ of the second-order cones is partitioned into four subsets $\Bar{\Bcal}, \Bar{\Ncal}, \Bar{\Rcal}$, and $\Bar{\Tcal}$, where $\Bar{\Tcal}$ is further partitioned to $\Bar{\Tcal} := (\Bar{\Tcal_1}, \Bar{\Tcal_2}, \Bar{\Tcal_3})$ as follows,
\begin{align*}
    \Bar{\Bcal} &:= \{ \ i \ | \ \xbar_1^i > ||\xbar_{2:n_i}^i||_2, \text{ for some } \xbar \in {\Pcal}_{SOCO}^{*} \},\\
    \Bar{\Ncal} &:= \{ \ i \ | \ \sbar_1^i > ||\sbar_{2:n_i}^i||_2, \text{ for some } \sbar \in {\Dcal}_{SOCO}^{*} \},\\
    \Bar{\Rcal} &:= \{ \ i \ | \ \xbar_1^i = ||\xbar_{2:n_i}^i||_2 > 0, \sbar_1^i = ||\sbar_{2:n_i}^i||_2 > 0,  \text{ for some } (\xbar; \ybar; \sbar) \in {\Pcal}_{SOCO}^{*} \times {\Dcal}_{SOCO}^{*} \},\\
    \Bar{\Tcal}_1 &:= \{\ i \ | \ \xbar^i = \sbar^i = 0, \text{ for all } (\xbar; \ybar; \sbar) \in {\Pcal}_{SOCO}^{*} \times {\Dcal}_{SOCO}^{*}\},\\
    \Bar{\Tcal}_2 &:= \{\ i \ | \ \sbar^i = 0, \text{ for all } (\ybar; \sbar) \in {\Dcal}_{SOCO}^{*}, \ {\rm and} \  \xbar_1^i = ||\xbar_{2:n_i}^i||_2 > 0, \text{ for some } \xbar \in {\Pcal}_{SOCO}^{*}\},\\
    \Bar{\Tcal}_3 &:= \{\ i \ | \ \xbar^i = 0, \text{ for all } \xbar \in {\Pcal}_{SOCO}^{*}, \ {\rm and} \ \sbar_1^i = ||\sbar_{2:n_i}^i||_2 > 0, \text{ for some } (\ybar; \sbar) \in {\Dcal}_{SOCO}^{*}\}.
\end{align*}
It should be highlighted that, due to the convexity of the optimal set, $\Bar{\Bcal}, \Bar{\Ncal}, \Bar{\Rcal}$, and $\Bar{\Tcal}$ are mutually disjoint and their union is the index set $\{1,...,r\}$. Therefore, it follows from the complementarity condition that for all $(\xtilde^*;\ytilde^*;\stilde^*) \in {\Pcal}_{SOCO}^* \times {\Dcal}_{SOCO}^*$, $\xtilde^i = 0$ for all $ i \in \Bar{\Ncal}$, and $\stilde^i = 0$ for all $ i \in \Bar{\Bcal}$ \cite{mohammad2019conic}.

For SDO, let $\Bcal := \Rcal(X^*)$ and $\Ncal := \Rcal(S^*)$, where $(X^*,y^*,S^*)$ is a maximally complementary optimal solution, meaning that we have $\Rcal(X) \subseteq \Bcal$ and $\Rcal(S) \subseteq \Ncal$ for all $(X,y,S) \in {\Pcal}_{SDO}^* \times {\Dcal}_{SDO}^*$. By the complementarity condition, the subspaces $\Bcal$ and $\Ncal$ are orthogonal. Moreover, let subspace $\Tcal$, be the orthogonal complement to $\Bcal+\Ncal$. The partition $(\Bcal, \Ncal, \Tcal)$ of $\Rmbb^n$ is called the optimal partition of an SDO problem. We can represent $X^*$ and $S^*$ using a common eigenvector basis, $Q^*$, as $X^* = Q^* \Lambda(X^*) (Q^*)^T$, and $S^* = Q^* \Lambda(S^*) (Q^*)^T$, where $\Lambda(X^*)$ and $\Lambda(S^*)$ corresponds to the diagonal matrices containing the  eigenvalues of $X^*$ and $S^*$, respectively. Thus, we have $\Rcal(X^*) = \Rcal(Q^* \Lambda(X^*))$, and $\Rcal(S^*) = \Rcal(Q^* \Lambda(S^*)).$ In particular, the columns of $Q^*$ corresponding to the positive eigenvalues of $X^*$ and $S^*$ can be chosen as an orthonormal basis for $\Bcal$ and $\Ncal$, respectively \cite{mohammad2019conic}.

Current literature
\cite{alizadeh2003second,andreani2022global,ben2001lectures,de2006aspects,dueri2014automated,jin2013exact,sim2005note,mohammad2019conic,wolkowicz2012handbook} suggests that a SOCO problem can be embedded in an SDO problem using the arrow-head matrix transformation,
\begin{equation} \label{eq:arw}
    \arw(\xbar^i) := 
    \begin{pmatrix}
    \xbar_1^i       & (\xbar_{2:n_i}^i)^T \\
    \xbar_{2:n_i}^i & \xbar_1^i I_{n_i -1}
    \end{pmatrix}
    \succeq 0
    \Leftrightarrow
    \xbar^i \in \Lcal^{n_i}.
\end{equation}
However, this transformation cannot be used to map both primal and dual solutions at the same time. Upon using the arrow-head representation of vectors $\xbar^i$ and $\sbar^i$ simultaneously, we might lose duality and complementarity. The following example illustrates that we may lose complementarity.
\begin{example} \label{exm:1}
Let $(\xbar;\ybar;\sbar)$ be an optimal solution of SOCO and assume that there exist at least one index $i \in \Rcal$. For all $ i \in \Rcal$, we can represent a solution as 
\begin{align*} \label{eq:partrsol}
    \xbar^i = \zeta^i \begin{pmatrix}
    1 \\
    \frac{\xbar_{2:n_i}^i}{||\xbar_{2:n_i}^i||_2}
    \end{pmatrix}
    ,
    \qquad
    \sbar^i = \xi^i \begin{pmatrix}
    1 \\
    \frac{\sbar_{2:n_i}^i}{||\sbar_{2:n_i}^i||_2}
    \end{pmatrix},
    \qquad
    \frac{\xbar_{2:n_i}^i}{||\xbar_{2:n_i}^i||_2} = - \frac{\sbar_{2:n_i}^i}{||\sbar_{2:n_i}^i||_2},
\end{align*}
where $\zeta^i = \xbar_1^i \geq 0$, $\xi^i = \sbar_1^i \geq 0$, and for at least one $(\xbar^*;\ybar^*;\sbar^*)$ we have both  $\zeta^i, \xi^i >0$. Without loss of generality, and for the sake of simplicity, assume that $\zeta^i = \xi^i = 1$. Moreover, let 
\begin{equation*}
    u_j = \frac{\xbar_{j}^i}{||\xbar_{2:n_i}^i||_2} = - \frac{\sbar_{j}^i}{||\sbar_{2:n_i}^i||_2}, \quad j = 2,...,n_i,
\end{equation*}
and $u = (u_2;...;u_{n_i})$. Then, using the arrow-head matrix transformation, we have
\begin{align*}
    X^i = \arw(\xbar^i) = 
    \begin{pmatrix}
    1   & u^T\\
    u & I_{n_i -1}
    \end{pmatrix},
    \qquad
    S^i = \arw(\sbar^i) = 
    \begin{pmatrix}
    1   & -u^T\\
    -u & I_{n_i -1}
    \end{pmatrix}.
\end{align*}
While $\xbar^i \circ \sbar^i = 0$, this transformation does not preserve complementarity as we have
\begin{equation*}
    X^i S^i =
     \begin{pmatrix}
    1   & u^T\\
    u & I_{n_i -1}
    \end{pmatrix}
    \begin{pmatrix}
    1   & -u^T\\
    -u & I_{n_i -1}
    \end{pmatrix} 
    =
    \begin{pmatrix}
    0            & [0]_{1\times(n-1)}\\
    [0]_{(n-1)\times1} & I_{n_i -1} - uu^T
    \end{pmatrix}
    \neq 0.
\end{equation*}
\end{example}
Example \ref{exm:1} shows that the arrow-head matrix transformation is not sufficient to represent a primal-dual pair of SOCO problems as an SDO problem. Thus, it seems worth exploring the actual relationship between an instance of SOCO and it's SDO counterpart.

To address this issue, Sim and Zhao \cite{sim2005note} started from a SOCO dual problem and exploited the arrowhead representation  \eqref{eq:arw} of the dual SOCO problem, to obtain the SDO dual as follows,
\begin{equation} \label{mo:szdual} \tag{$\Dcal_{\rm SZ}$}
    \begin{aligned}
    \max &\ b^T y \\
    \st  &\ \sum_{i=1}^{m} y_i \arw(a_{(i)}^j) + S^j = \arw(c^j) \text{ for all } j = 1,...,r,&& \\
         &\ S^j \succeq 0\text{ for all } j = 1,...,r,                         && 
    \end{aligned}
\end{equation}
where $a_{(i)}^j$ denotes $i^{\rm th}$ row of the matrix $A$ corresponding to Lorentz cone $j$. Observe that $S^j$ as a linear combination of arrow-head matrices is an arrow-head matrix, too. Using this dual model, we get the SDO primal problem as
\begin{equation} \label{mo:szprimal} \tag{$\Pcal_{\rm SZ}$}
    \begin{aligned}
    \min &\ \sum_{j=1}^{r} \trace( \arw(c^j) X^j) \\
    \st  &\ \sum_{j=1}^{r}\trace(\arw(a_{(i)}^j) X^j) = b_i \text{ for all }i = 1,...,m, && \\
         &\ X^j \succeq 0\text{ for all } j = 1,...,r       .                    
    \end{aligned}
\end{equation}
They showed that $X^j = \arw(\xbar^j)$ is not a feasible solution for the SDO primal problem \eqref{mo:szprimal}. In fact, primal feasible solutions of \eqref{mo:szprimal} are fully dense, and do not have an arrow-head structure. To fix this issue, they proposed the mapping 
\begin{equation} \label{eq:sim-zhao-map}
{\rm MR}(\xbar^j)=
\begin{bmatrix}
    \frac{1}{4}\theta^j     & \frac{1}{2}(\xbar_{2:n}^j )^T\\
    \frac{1}{2}\xbar_{2:n}^j& \frac{\xbar_1^j-\|\xbar_{2:n}^j\|}{2(n-1)}I+\frac{\xbar_{2:n}^j(\xbar_{2:n}^j)^T}{\theta^j}
\end{bmatrix},
\end{equation}
where $\theta^j=\xbar_1^j+\|\xbar_{2:n}^j\|+\sqrt{(\xbar^j_1+\|\xbar_{2:n}^j\|)^2-4\|\xbar_{2:n}^j\|^2}$.  

In our study, a key concept is the notion of admissible map which is defined next. 
\begin{definition} \label{def:admismap}
A mapping $\Mcal$ is called admissible if it preserves feasibility and objective function value, i.e.
\begin{align*}
    (\xbar,\ybar,  \sbar) \in \Fcal_{\Pcal_{SOCO}} \times \Fcal_{\Dcal_{SOCO}}  &\Rightarrow  \Mcal(\xbar,\ybar, \sbar) \in \Fcal_{\Pcal_{SDO}} \times \Fcal_{\Dcal_{SDO}},\\
    (X,y,S) \in \Fcal_{\Pcal_{SDO}} \times \Fcal_{\Dcal_{SDO}} &\Rightarrow  \Mcal^{-1}(X,y,S) \in \Fcal_{\Pcal_{SOCO}} \times \Fcal_{\Dcal_{SOCO}},\\
    c^T \xbar = \trace{(CX)}&, \
    b^T \ybar = b^T y.  
\end{align*} 
\end{definition}
The mapping of Sim and Zhao \cite{sim2005note} is admissible, and they proved that it maps a solution from the boundary (interior) of the Lorentz cone to a solution on the boundary (interior) of the cone of semidefinite  matrices. In this paper, we seek to extend their approach and explore mappings that satisfy the definition of admissible mapping. Although the mapping of \cite{sim2005note} is a point to point map, the image of $(\xbar,\ybar,  \sbar)$ might be a point or a set. In addition, the SDO representation of SOCO of \cite{sim2005note}, \eqref{mo:szprimal} and \eqref{mo:szdual}, is defined using the product of multiple cones of positive semidefinite matrices, but we use a more general approach to get an SDO representation in standard form. The major goal of this paper is clarifying more the relationship between SOCO and the related SDO by developing different mappings and exploring the relationship between the optimal partitions of these problems.   

Without loss of generality, in Sections \ref{sec:3} and \ref{sec:4}, we first present the results in case of a single second-order cone, and then we  generalize the results to the multiple cone case. To this end, we consider the following primal and dual problems,
\begin{align}
    &z_{\Pcal^1_{SOCO}}^*=\min \big\{ c^T \xbar \ : \ A\xbar =b, \ \xbar \in \Lcal^n \big\}, \label{mo:socop} \tag{$\Pcal^1_{SOCO}$} \\
    &z_{\Dcal^1_{SOCO}}^*=\max \big\{ b^T \ybar \ : \ A^T \ybar + \sbar = c, \ (\ybar,\sbar) \in \mathbb{R}^m\times\Lcal^n \big\}, \label{mo:socod} \tag{$\Dcal^1_{SOCO}$}
\end{align}
with feasible sets $\Fcal_{\Pcal^1_{SOCO}} =\{ \xbar \in \Lcal^n: A\xbar=b\}$ and $\Fcal_{\Dcal^1_{SOCO}}=\{ (\ybar,\sbar) \in \mathbb{R}^m\times\Lcal^n: A^T \ybar + \sbar = c\}$, and optimal solution sets ${\Pcal^1_{SOCO}}^* = \{ \xbar \in \Fcal_{\Pcal^1_{SOCO}}: c^T\xbar=z_{\Pcal^1_{SOCO}}^*\}$ and ${\Dcal^1_{SOCO}}^*=\{ (\ybar,\sbar) \in \Fcal_{\Dcal^1_{SOCO}}: b^T\ybar=z_{\Dcal^1_{SOCO}}^*\}$, respectively. 

\section{From SOCO to SDO: Starting from the Dual Side} \label{sec:3}
One can derive the SDO counterpart of a SOCO problem starting with either the primal or dual SOCO problem. In this section, similar to \cite{sim2005note}, we initiate the derivation from the dual side of SOCO. Thus, as mentioned earlier, we preserve the arrow-head structure of the matrix $S$ corresponding to dual solution $\sbar$.
\subsection{Derivation and Solution Mapping}
We utilize the arrow-head transformation to the vector $c$ and the rows of matrix $A$,
\begin{align*}
    \Cvec = \arw(c),
    % \begin{bmatrix}
    % c_1    & c_2 &  ...   & c_n \\
    % c_2    & c_1 &        &     \\
    % \vdots &     & \ddots &    \\
    % c_n    &     &        & c_1 
    % \end{bmatrix}, 
    \qquad
    \Avec_i = \arw(a_{(i)}), \quad
    % \begin{bmatrix}
    % a_{i1} & a_{i2} & ...    & a_{in} \\
    % a_{i2} & a_{i1} &        &        \\
    % \vdots &        & \ddots &        \\
    % a_{in} &        &        & a_{i1} 
    % \end{bmatrix}, 
    \ i = 1,2, ..., m.
\end{align*}
Since in the SOCO dual $s = c - A^T \ybar$, by applying the arrow-head structure to $A$ and $c$, we have that  $S = \Cvec - \sum_{i=1}^m \ybar_i \Avec_i$ has the arrow-head structure, as it is a linear combination of arrow-head matrices. Therefore, the SDO counterpart of the SOCO dual problem \eqref{mo:socod} is as follows,
\begin{align*} \label{mo:dsdual}
    z^*_{\Dcal^{D}_{SDO}} := \max \left\{ b^T y \ :  \sum_{i=1}^m y_i\Avec_i + S = \Cvec, \  S \succeq 0 \right\}, \tag{$\Dcal^{D}_{SDO}$} 
\end{align*}
which has the following dual,
\begin{equation*} \label{mo:dsprimal}
    z^*_{\Pcal^{D}_{SDO}} := \min \ \left\{ \text{ Tr}(\Cvec X) \ 
    : \  \trace(\Avec_iX) = b_i, \quad i = 1,...,m, \quad X \succeq 0 \right\}, \tag{$\Pcal^{D}_{SDO}$}
\end{equation*}
as it's SDO primal problem.
% To make sure that the derived SDO primal represents SOCO primal correctly, the following variable change is required,
% \begin{align} \label{rel:map1}
%     \xbar = 
%     \begin{bmatrix}
%     \sum_{i=1}^n x_{ii},
%     2x_{12}         ,
%     \hdots          ,
%     2x_{1n}            
%     \end{bmatrix}^T.
% \end{align}
For the SDO problems \eqref{mo:dsprimal} and \eqref{mo:dsdual}, let
\begin{align*}
    \Fcal_{\Dcal^{D}_{SDO}} &= \{(y,S) \in \Rmbb^m \times \Smbb^n : \sum_{i=1}^m y_i \Avec_i + S = \Cvec,S \succeq 0\},\\
    \Fcal_{\Pcal^{D}_{SDO}} &= \{ X \in \Smbb^n : \trace(\Avec_i X) = b_i, i = 1,...,m, X \succeq 0\}
\end{align*}
represent the feasible sets, and 
\begin{align*}
    {\Dcal^{D}_{SDO}}^* = \{(y,S) \in \Fcal_{\Dcal^{D}_{SDO}}: b^T y = z^*_{\Dcal^{D}_{SDO}}\}, \qquad
    {\Pcal^{D}_{SDO}}^* = \{ X \in \Fcal_{\Pcal^{D}_{SDO}} : \trace{(\Cvec X)} = z^*_{\Pcal^{D}_{SDO}}\},
\end{align*}
represent the optimal solution sets, respectively.
% \begin{definition}
% A mapping $\Mcal$ is called admissible if it preserves feasibility and objective function value, i.e.
% \begin{align*}
%     (\xbar,\ybar, \sbar) \in \Fcal_{\Bar{\Pcal}}\times \Fcal_{\Bar{\Dcal}}  &\Rightarrow  \Mcal(\xbar,\ybar, \sbar) \in \Fcal_{\Pcal'} \times \Fcal_{\Dcal'},\\
%     (X,y,S) \in \Fcal_{\Pcal'} \times \Fcal_{\Dcal'} &\Rightarrow  \Mcal^{-1}(X,y,S) \in \Fcal_{\Bar{\Pcal}}\times \Fcal_{\Bar{\Dcal}},\\
%     c^T \xbar = \trace{(\Cvec X)}&, \
%     b^T \ybar = b^T y.  
% \end{align*} 
% \end{definition}
The following theorem provides a point to set admissible mapping, see Definition \ref{def:admismap}  for $r=1$, based on the \eqref{mo:socod} and \eqref{mo:socop}, and their representations \eqref{mo:dsdual}, \eqref{mo:dsprimal}.
\begin{theorem} \label{th:1}
Consider the SOCO problem pairs (\ref{mo:socop}) and (\ref{mo:socod}) with $(\xbar, \ybar, \sbar) \in \Fcal_{\Pcal^1_{SOCO}} \times \Fcal_{\Dcal^1_{SOCO}}$, and the SDO problem pairs (\ref{mo:dsprimal}) and (\ref{mo:dsdual}) with $(X,y,S) \in \Fcal_{\Pcal^{D}_{SDO}}\times \Fcal_{\Dcal^{D}_{SDO}}$. Then, mapping $(X,y,S) = \Mcal(\xbar,\ybar,\sbar)$ with
\begin{align*}
    S = {\rm Arw}(\sbar), \qquad
    y = \ybar,            \qquad
    X = 
    \begin{bmatrix}
    X_{11} & X_{12} & \hdots & X_{1n}\\
    X_{12} & X_{22} & \hdots & X_{2n}\\
    \vdots & \vdots & \ddots & \vdots\\
    X_{1n} & X_{2n} & \hdots & X_{nn}
    \end{bmatrix}
    \succeq 0 \  {\rm with} \
    \begin{bmatrix}
    \sum_{i=1}^n X_{ii}  \\
    X_{12}               \\
    \vdots               \\
    X_{1n}            
    \end{bmatrix} 
    = \begin{bmatrix}
    \xbar_1\\
    \frac{\xbar_2}{2}\\
    \vdots\\
    \frac{\xbar_n}{2}
    \end{bmatrix}, 
\end{align*}
is a point-to-set admissible mapping. In addition, the inverse mapping denoted by $(\xbar,\ybar,\sbar) = \Mcal^{-1}(X,y,S)$, with 
\begin{align*}
    \sbar = {\rm Arw}^{-1}(S), \qquad
    \ybar = y, \qquad
    \xbar = 
    \begin{bmatrix}
    \sum_{i=1}^n X_{ii},  
    2X_{12},              
    \hdots,               
    2X_{1n}        
    \end{bmatrix}^T,
\end{align*}
is a point-to-point admissible mapping.
\end{theorem}
\begin{proof}
The proof of this theorem is presented in Appendix \ref{appsec:1}.
\end{proof}

The following corollaries restate that the provided mapping preserves the objective function value.
\begin{corollary} \label{col:d-obj}
We have $z_{\Pcal^1_{SOCO}}^* = z^*_{\Pcal^{D}_{SDO}}$, and $z_{\Dcal^1_{SOCO}}^* = z^*_{\Dcal^{D}_{SDO}}$.
\end{corollary}
\begin{corollary} \label{col:d-opt}
A feasible solution $(\xbar,\ybar,\sbar) \in \Fcal_{\Pcal^1_{SOCO}} \times \Fcal_{\Dcal^1_{SOCO}}$ is optimal for a pair of SOCO problems \eqref{mo:socop}, and \eqref{mo:socod} with optimal value $(z^*_P,z^*_D)$ if and only if the mapped solution $(X,y,S)$ is optimal for the SDO problems \eqref{mo:dsprimal}, and \eqref{mo:dsdual} with optimal value $(z^*_P,z^*_D)$  . 
\end{corollary}
\begin{corollary} \label{col:d-comp}
A feasible solution $(\xbar,\ybar,\sbar) \in \Fcal_{\Pcal^1_{SOCO}} \times \Fcal_{\Dcal^1_{SOCO}}$ is optimal for a pair of SOCO problems \eqref{mo:socop}, and \eqref{mo:socod} with zero duality gap if and only if the mapped solution $(X,y,S)$ is optimal for the SDO problems \eqref{mo:dsprimal}, and \eqref{mo:dsdual} with zero duality gap, i.e., 
$$\xbar \circ \sbar=0\iff \trace{(XS)}=0$$
\end{corollary}
Note that these results are valid regardless of duality (strong / weak with gap), status of the SOCO problems. Furthermore, observe that one can propose different admissible mappings that satisfies the conditions presented in Theorem \ref{th:1}. In Section \ref{sec:r1map}, we propose a rank-one mapping, which is the simplest option. In Section \ref{ssec:FRM}, we show that when $\xbar \in {\rm int}(\Lcal^n)$, full rank mappings can also be obtained, which map a solution in the interior of SOCO to a solution in the interior of the SDO cone.

\subsection{Rank-one Mapping} \label{sec:r1map}
% Based on Theorem \ref{th:1}, we have an explicit formula for finding the mapped solution in any direction, except calculating $X$ for any $\xbar\in \Fcal_{\bar{\Pcal}}$.
In this section, we construct a rank-one matrix $X$ for vector $\xbar$ that satisfies the conditions in Theorem \ref{th:1}. Thus, we introduce the vector $\beta\in\Rmbb^n$, and define the rank-one matrix.
\begin{align*}
    X = \beta \beta^T, 
    % =
    % \begin{bmatrix}
    % \beta_1^2       & \beta_1 \beta_2 & \hdots & \beta_1 \beta_n \\
    % \beta_1 \beta_2 & \beta_2^2       & \hdots & \beta_2 \beta_n \\
    % \vdots          & \vdots          & \ddots & \vdots \\
    % \beta_1 \beta_n & \beta_2 \beta_n & \hdots & \beta_n^2
    % \end{bmatrix},
\end{align*}
with
\begin{align*}
    \sum_{i=1}^n {\beta}_i^2 &= \xbar_1,\\
    \beta_1 \beta_j &= \frac{\xbar_j}{2} \text{  for all } j = 2, ..., n.
\end{align*}
We need to solve this $n$-variable-$n$-equation system. The solution of this system is $\beta = 0$ if $\xbar = 0$, and if $\xbar \neq 0$ then
\begin{equation*}
    \beta = \frac{1}{\sqrt{2(\xbar_1 + \delta)}}(\xbar_1 + \delta,\xbar_2, ..., \xbar_n)^T,
\end{equation*}
where $\delta = \sqrt{(\xbar_1)^2 - ||\xbar_{2:n}||^2}$. It is easy to see that if we are on the boundary of the second-order cone, then $\delta = 0$. Otherwise, we have $\delta \neq 0$. Using this vector, we can construct a suitable matrix $X$. 

\begin{theorem} \label{th:2} 
Consider the rank-one mapping with
\begin{align} \label{mo:r1map}
    X =
    \begin{cases}
    [0]_{n \times n} & \text{ if } \xbar = (0,0,...,0),\\
    {\rm DMR}^{1}(\xbar)              & \text{ otherwise}.
    \end{cases}
\end{align}
where
\begin{equation*}
    {\rm DMR}^{1}(\xbar) = \beta \beta^T =  
    \begin{bmatrix}
    \frac{\xbar_1 + \delta}{2} & \frac{\xbar_2}{2} & \hdots & \frac{\xbar_n}{2} \\
    \frac{\xbar_2}{2} & \frac{\xbar_2^2}{2[\xbar_1 + \delta]} & \hdots & \frac{\xbar_2 \xbar_n}{2[\xbar_1 + \delta]} \\
    \vdots    & \vdots    & \ddots & \vdots    \\
    \frac{\xbar_n}{2} & \frac{\xbar_2 \xbar_n}{2[\xbar_1 + \delta]} & \hdots & \frac{\xbar_n^2}{2[\xbar_1 + \delta]} \\
    \end{bmatrix}.
\end{equation*}
Then, (\ref{mo:r1map}) together with $(y,S) = (\ybar,{\rm Arw}(\sbar))$ is a point-to-point admissible mapping.
\end{theorem}

\begin{proof}
The proof is straightforward as it is enough to show that matrix ${\rm DMR}^{1}(\xbar)$ satisfies the conditions in Theorem \ref{th:1}.
\end{proof}

\subsection{Higher Rank Mapping} \label{ssec:FRM}
In this section, we show that when a SOCO solution is in the interior of the cone, i.e. $\xbar \in {\rm int}(\Lcal^n)$, then we can use a full rank mapping, i.e.
\begin{equation*}
    {\rm DMR}^{n}(\xbar)= \sum_{i=1}^{n}\beta^i(\beta^i)^T.
\end{equation*}

\begin{theorem} \label{th:3}
There exist mappings  ${\rm DMR}$ where {\rm rank}$({\rm DMR}(\bar{x}))=n$ for $\xbar\in {\rm int}(\Lcal^n)$.
\end{theorem}
\begin{proof}
The proof of this theorem is presented in Appendix \ref{appsec:2}. 
\end{proof}
One can easily modify Algorithm~\ref{alg:FRM} (presented in proof of Theorem \ref{th:3}) to map a solution $\xbar \in {\rm int}(\Lcal^n)$ to a matrix ${\rm DMR}^{k}(\xbar)$
with rank $1\leq k\leq n$. This means that the primal feasible set of a SOCO can be mapped to different subsets of the SDO primal feasible region, e.g., rank-one mapping maps the primal feasible set of a SOCO to a one-dimensional face of the SDO primal feasible set.

Recall that Theorem 1 of \cite{sim2005note} proposes the mapping
$${\rm MR}(\xbar)=\begin{pmatrix}
\frac{1}{4}\theta&\frac{1}{2}\xbar_{2:n}^T\\
\frac{1}{2}\xbar_{2:n}&\frac{\xbar_1-\|\xbar_{2:n}\|}{2(n-1)}I+\frac{\xbar_{2:n}\xbar_{2:n}^T}{\theta}
\end{pmatrix},$$
where $\theta=\xbar_1+\|\xbar_{2:n}\|+\sqrt{(\xbar_1+\|\xbar_{2:n}\|)^2-4\|\xbar_{2:n}\|^2}$. They showed that this map is admissible. Moreover, it has full rank when $\xbar_1>\|\xbar_{2:n}\|$, i.e. $\xbar \in {\rm int}(\Lcal^n)$, and it is a rank-one matrix when $\xbar_1=\|\xbar_{2:n}\|$. This map also proves Theorem \ref{th:3}, while our proof follows a different approach. In our approach, we can generate different mappings and explore the feasible set of SDO by changing parameter $\epsilon$ in Algorithm~\ref{alg:FRM}. However, to prove that the set of all maps with different rank  produced by our approach can build the whole feasible set of the SDO representation is an ongoing research.  
%%%%%%%%%%%%%%%%%%%%%%%%%%%%%%%%

Given the mapping of Sim and Zhao \cite{sim2005note}, consider the case in which the solution is on the boundary of the second-order cone, i.e. $\xbar_1 = \|\xbar_{2:n}\|$. Then, $\theta = 2\xbar_1$, and 
$${\rm MR}^1 (\xbar)=\begin{pmatrix}
\frac{1}{2}\xbar_1&\frac{1}{2}\xbar_{2:n}^T\\
\frac{1}{2}\xbar_{2:n}&\frac{\xbar_{2:n}\xbar_{2:n}^T}{2\xbar_1}
\end{pmatrix}.$$
We can see that this is exactly identical to the rank one mapping we presented in Theorem \ref{th:2}. We can write it as 
\begin{equation*}
    {\rm MR}^1 (\xbar)= \nu^1 (\nu^1)^T,
\end{equation*}
where 
\begin{equation*}
    \nu^1 = \frac{1}{\sqrt{\theta}} \left( \frac{1}{2} \theta, \xbar_{2}, \dots, \xbar_{n} \right)^T.
\end{equation*}
In order to construct ${\rm MR}(\xbar)$ as the sum of rank-one matrices, we define
\begin{equation*}
    \nu^j = \sqrt{\frac{\xbar_1 - \|\xbar_{2:n} \|}{2(n-1)}} \ e_j \qquad {\rm for } \ \ j =2,\dots,n,
\end{equation*}
where $e_j$ is a unit vector with 1 in element $j$. By this setting, we have 
$$
{\rm MR}(\xbar)= \sum_{j=1}^{n}\nu^j(\nu^j)^T.
$$ 
If the solution $\xbar$ of the SOCO primal problem is on the boundary of the cone, i.e. $\xbar_1 = \|\xbar_{2:n} \|$, then we get
\begin{align*}
    \nu^1 &= \frac{1}{\sqrt{2\xbar_1}} \left(\xbar_1, \xbar_{2}, \dots, \xbar_{n} \right)^T, \\
    \nu^j &= 0, \qquad \qquad \text{  for } j = 2,\dots, n.
\end{align*}
This results in the rank one mapping ${\rm MR}^1 (\xbar)$. On the other hand, if the solution is in the interior of the cone, i.e. $\xbar_1 > \|\xbar_{2:n} \|$, then $\nu^j \neq 0$ for $j = 2,\dots, n$, and we can construct the full rank mapping ${\rm MR}(\xbar)$. Note that one cannot take a combination of first $k$ vectors $\nu^j$ to construct a rank-$k$ mapping as it would \textit{not} be admissible since it violates the conditions given in Theorem \ref{th:1}, i.e. sum of diagonals will not be equal to $\xbar_1$. To overcome this issue and construct a rank-$k$ mapping, let $\Ncal \subseteq \{2,\dots,n\}$ with $|\Ncal| = k$. One can take the following definition of $\nu^j$,
\begin{equation*}
    \nu^j = \sqrt{\frac{\xbar_1 - \|\xbar_{2:n} \|}{2(k-1)}} \ e_j \qquad {\rm for } \ \ j \in \Ncal, \ {\rm and } \ \nu^j = 0 \ {\rm for } \ \ j \notin \Ncal.
\end{equation*}
The resulting matrix ${\rm MR}^k(\xbar)$ has rank $k$, and one can easily see that it satisfies the conditions of Theorem \ref{th:1}.
Considering $k = n$, this choice of $\nu^j$ results in identical mapping to ${\rm MR}(\xbar)$.
Next theorem shows that we can have mappings with different ranks when we are mapping a solution from interior of the Lorentz cone.
%%%%%%%%%%%%%%%%%%%%%%%%
\begin{theorem}\label{th:4}
Let $\rho(\xbar)=\max\{\rank (\Mcal(\xbar)): \Mcal\text{ is admissible map} \}$. We have 
\begin{itemize}
    \item $\rho(\xbar)=n$ if $\xbar\in {\rm int}(\Lcal^n).$
    \item $\rho(\xbar)=1$ if $\xbar\in \partial(\Lcal^n).$
\end{itemize}
\end{theorem}
\begin{proof}
Proof of this theorem is similar to Theorem 1 of \cite{sim2005note}.
\end{proof}
In the next section, we generalize our result for the case there are multiple Lorentz cones.
\subsection{Generalization to Multiple SOCs} \label{ss:dsgen}
In this section, we extend our rank-one mapping to the case with multiple second-order cones. We can use similar conditions as in Theorem \ref{th:1} to extend our mapping to the case of multiple second-order cones. Although, instead of the ``Arw" operator, we need to introduce a new operator called ``DArw" which constructs a block diagonal matrix with arrow-head matrices of input vectors. First, we transform the objective coefficient vectors and coefficient matrices into proper block-diagonal structure as
\begin{align*}
    \Ctilde = \text{DArw}(c^1, c^2, ..., c^r) = 
    \begin{bmatrix}
    \Cvec^1 &           &        &            \\
            & \Cvec^2   &        &            \\
            &           & \ddots &            \\
            &           &        & \Cvec^r
    \end{bmatrix},
\end{align*}
where $\Cvec^i$  is the arrow-head matrix corresponding to vector $c^i$ for all $i = 1,...,r$. Moreover, we define a similar block diagonal matrix 
\begin{align*}
    \Atilde_j = \text{DArw}(A_j^1, A_j^2, ..., A_j^r) = 
    \begin{bmatrix}
    \Avec_j^1 &             &        &             \\
                & \Avec_j^2 &        &             \\
                &             & \ddots &             \\
                &             &        & \Avec_j^r
    \end{bmatrix} \qquad \qquad \text{ for all } j= 1,...,m,
\end{align*}
where $A_j^i$ corresponds to the row $j$ of matrix $A^i$ for all $i = 1,...,r$. Now, the SDO representation of the SOCO problem can be derived as (\ref{mo:gensdop}) and (\ref{mo:gensdod}). Compared to the SDO representation of \cite{sim2005note}, \eqref{mo:szprimal} and \eqref{mo:szdual}, our representation is a standard SDO over one cone of positive semidefinite matrices of dimension $n\times n$. Given the introduced notations $\Ctilde$ and $\Atilde$, we have 
\begin{equation} \label{mo:gendualp} \tag{$\Pcal_{SDO}^{\Dtilde}$}
    \begin{aligned}
     z_{\Pcal_{SDO}^{\Dtilde}}^* := \min  \ & \trace{(\Ctilde \Xtilde)} \\
     \text{ s.t.} \ & \trace{(\Atilde_i \Xtilde)} = b_i && \text{ for all } i = 1,...,m, \\
     \ & \quad \Xtilde \succeq 0, 
    \end{aligned}
\end{equation}
and the dual problem is
\begin{equation} \label{mo:genduald} \tag{$\Dcal_{SDO}^{\Dtilde}$} 
    \begin{aligned}
    z_{\Dcal_{SDO}^{\Dtilde}}^* := \max \ &  b^T y \\
    \text{ s.t.} \ & \sum_{i=1}^m y_i \Atilde_i + \Stilde = \Ctilde, \\
    \ & \Stilde \succeq 0.
    \end{aligned}
\end{equation}
Analogously, we can define the sets of feasible and optimal solutions for the problems \eqref{mo:gendualp} and \eqref{mo:genduald}.

We can present the following theorem to specify admissible mappings for the general form. 
\begin{theorem} \label{th:5}
Consider the SOCO problem pairs (\ref{mo:gensocop}) and (\ref{mo:gensocod}) with 
\begin{equation*}
    (\xbar^1;\xbar^2;...;\xbar^r;\ybar;\sbar^1;\sbar^2;...;\sbar^r) \in \Fcal_{\Pcal_{SOCO}} \times \Fcal_{\Dcal_{SOCO}},
\end{equation*}
and SDO problem pairs \eqref{mo:gendualp} and \eqref{mo:genduald} with $(\Xtilde,\ytilde,\Stilde) \in \Fcal_{\Pcal_{SDO}^{\Dtilde}} \times \Fcal_{\Dcal_{SDO}^{\Dtilde}}$. Then, the following mapping, denoted by $(\Xtilde,\ytilde,\Stilde) = \Mcal(\xbar^1;\xbar^2;...;\xbar^r;\ybar;\sbar^1;\sbar^2;...;\sbar^r)$, with
\begin{align*}
    \Stilde &= \text{DArw}(\sbar^1;...;\sbar^r), \\
    \ytilde &= \ybar, \qquad \\
    \Xtilde &
    \succeq 0 \quad {\rm s.t.} \quad 
    \begin{bmatrix}
    \sum_{j=u_i}^{u_i+n_i} \Xtilde_{jj}  \\
    \Xtilde_{u_i,u_i+1}                  \\
    \vdots                               \\
    \Xtilde_{u_i,u_i+n_i}            
    \end{bmatrix} 
    = 
    \begin{bmatrix}
    \xbar_1^i\\
    \frac{\xbar_2^i}{2}\\
    \vdots\\
    \frac{\xbar_n^i}{2}
    \end{bmatrix},  \text{ for all i = 1,...,r}
\end{align*}
coupled with the inverse mapping denoted by $(\xbar^1;\xbar^2;...;\xbar^r;\ybar;\sbar^1;\sbar^2;...;\sbar^r) = \Mcal^{-1}(\Xtilde,\ytilde,\Stilde)$ with 
\begin{align*}
    (\sbar^1;...;\sbar^r) = \text{DArw}^{-1}(\Stilde), \qquad
    \ybar = \ytilde, \qquad
    (\xbar^1;...;\xbar^r) = 
    \begin{bmatrix}
    \sum_{j=u_i}^{u_i+n_i} \Xtilde_{jj}  \\
    2\Xtilde_{u_i,u_i+1}                                                       \\
    \vdots                                                          \\
    2\Xtilde_{u_i,u_i +n_i}            
    \end{bmatrix}  ,
\end{align*}
where $u_i = 1+\sum_{k=0}^{i-1} n_{i}$ and $n_0 = 0$, is an admissible mapping.
\end{theorem}
\begin{proof}
The proof of this theorem is analogous to that of Theorem \ref{th:1}.
\end{proof}
\begin{remark}
    Note that a solution of (\ref{mo:gendualp}) is not necessarily block-diagonal.
\end{remark}

Using the conditions in Theorem \ref{th:4}, we can write
\begin{align*}
    \Stilde &= \text{DArw}(\sbar^1, \sbar^2, ..., \sbar^r)
    = 
    \begin{bmatrix}
    \arw(\sbar^1) &                     &        &                     \\
                        & \arw(\sbar^2) &        &                     \\
                        &                     & \ddots &                     \\
                        &                     &        & \arw(\sbar^r)
    \end{bmatrix}, \quad \text{ and, }\\
    \ytilde &= \ybar.
\end{align*}
We can define matrix $\Xtilde$ using our rank one mapping with vector $\Tilde{\beta}$. Vector $\Tilde{\beta}$ consists of $n$ elements partitioned into $r$ sub-vectors each corresponds to the cones in the problem. Thus, we have
\begin{equation*}
    \Xtilde = \Tilde{\beta} \Tilde{\beta}^T,
\end{equation*}
where
\begin{align*}
    \Tilde{\beta} = ((\beta^1)^T, (\beta^2)^T, ..., (\beta^r)^T)^T.
\end{align*}
% \begin{align*}
%     \beta = (\beta^1_1, \beta^1_2, ..., \beta^1_{n_1}, \beta^2_1, \beta^2_2, ..., \beta^2_{n_2}, ..., \beta^r_1, \beta^r_2, ..., \beta^r_{n_r})^T.
% \end{align*}
One can calculate vector $\Tilde{\beta}$ in the closed form as follows,
\begin{align*}
    % \Bar{\beta}'^i = 
    % \begin{cases} 
    % \sqrt{\frac{\xbar^i_1 + \Bar{\delta}'^i}{2}} & j=1,\\
    % \frac{\xbar^i_j}{\sqrt{2(\xbar^i_1 + \Bar{\delta}'^i)}} & j=2,...,n_i,
    % \end{cases}
    \beta^i = \frac{1}{\sqrt{2(\xbar^i_1 + \delta^i)}}\left( \xbar^i_1 + \delta^i, \xbar^i_2, \dots, \xbar^i_{n_i} \right),
\end{align*}
where $\delta^i = \sqrt{(\xbar^i_1)^2 - ||\xbar^i_{2:n}||^2}$.
 We can also have a separate rank-one mapping for each Lorentz cone as $\Xtilde = \sum_{i=1}^{r}(\beta)_i (\beta)_i^T$, where
\begin{align*}
(\beta^j)_i =
    \begin{cases}
        \frac{1}{\sqrt{2(\xbar^i_1 + \delta^i)}}\left( \xbar^i_1 + \delta^i, \xbar^i_2, \dots, \xbar^i_{n_i} \right)&\text{ if } j=i,\\
        0 &\text{ otherwise}.
    \end{cases}
\end{align*}
In this case, $\Xtilde$ is a block diagonal matrix with rank equal to $r$. If we choose $(\beta^j)_i \in \mathbb{R}^{n_j}$ for $j\in\{1,\dots,r\}/\{i\}$, then $\Xtilde = \sum_{i=1}^{r}(\beta)_i (\beta)_i^T$ will be a positive semidefinite matrix, which is not necessarily block diagonal. Since all the input matrices are block diagonal, it is straightforward to check that the mapping will remain admissible. For any Lorentz cone $i$ with $\xbar_1^i>\|\xbar_{2:n}^i\|$, the corresponding part of the matrix $\Xtilde $ can have any rank from $1$ to $n_i$ using higher rank mappings proposed in the previous section. For instance, if the solution $\xbar$ is in the interior of all Lorentz cones, i.e. $\xbar_1^i>\|\xbar_{2:n}^i\|$ for $i=1,\dots,r$, then we can find $\Xtilde = \sum_{i=1}^{r}\sum_{k=1}^{n_r}(\beta)_i^k (\beta)_i^k{}^T$ whose rank is $n\times r$. In this case, it is straightforward to produce mapping with arbitrary rank form $1$ to $n\times r$ by using Algorithm~\ref{alg:FRM} for filling diagonal blocks and fill other parts with arbitrary numbers.

\section{From SOCO to SDO: Starting from the Primal Side} \label{sec:4}
In this section, we analyze the case when we start to reformulate the standard primal SOCO as an equivalent SDO problem, which requires more complex reformulation and was left untouched by Sim and Zhao \cite{sim2005note}. The purpose is to investigate the SOCO-SDO relationship starting from the primal side and answer the following questions.
\begin{enumerate}
    \item What happens if we force matrix $X$ to be arrow-head?
    \item Does this derivation results in similar mappings between the SOCO problems and their SDO counterparts?
\end{enumerate}
\subsection{Derivation and Solution Mapping}
Recall SOCO problems \eqref{mo:socop} and \eqref{mo:socod} and their corresponding feasible and optimal solution sets from Section \ref{sec:prelim}. Recall that we have $\xbar \in \Lcal^n$ if and only if $X=\arw(\xbar) \succeq 0$, thus starting from \eqref{mo:socop} requires a different choice of $\Cvec$ and $\Avec_i$ in order to properly represent the primal constraint of SOCO in an equivalent SDO reformulation. Thus, we define
\begin{align*}
    \Vec{\mathtt{C}} &= 
    % \begin{bmatrix}
    % \frac{1}{n} c_1    & \frac{1}{2}c_2 &  ...   & \frac{1}{2}c_n \\
    % \frac{1}{2} c_2    & \frac{1}{n}c_1 &        &     \\
    % \vdots &     & \ddots &    \\
    % \frac{1}{2} c_n    &     &        & \frac{1}{n} c_1 
    % \end{bmatrix},\\
    % &=
    \arw(\frac{1}{n} c_1, \frac{1}{2}c_2, ..., \frac{1}{n} c_n),\\
    \Vec{\mathtt{A}}_i &= 
    % \begin{bmatrix}
    % \frac{1}{n}a_{i1} & \frac{1}{2}a_{i2} & ...    & \frac{1}{2}a_{in} \\
    % \frac{1}{2}a_{i2} & \frac{1}{n}a_{i1} &        &                   \\
    % \vdots            &                   & \ddots &                   \\
    % \frac{1}{2}a_{in} &                   &        & \frac{1}{n}a_{i1} 
    % \end{bmatrix},\\
    % &=
    \arw(\frac{1}{n}a_{i1}, \frac{1}{2}a_{i2}, ..., \frac{1}{2}a_{in}), && i = 1,2, ..., m.
\end{align*}
which leads to preserving $a_{(i)}^T \xbar = \trace(\Vec{\mathtt{A}}_i X) = b_i, \ i=1,\dots,m$, where $X = \arw(\xbar)$. Using the just introduced arrow-head representations of $c$, $a_{(i)}$, and $\xbar$, we write the following SDO problem.
\begin{align*}
    \min \left\{\trace(\Vec{\mathtt{C}}  X) \ : \ \text{Tr}(\Vec{\mathtt{A}}_i  X)  = b_i, \quad i=1,...,m, \quad X \succeq 0 \right\}.
\end{align*}
However, in this SDO problem $X$ is a semidefinite matrix without arrow-head structure. Thus, in order to represent \eqref{mo:socop} using this SDO model, we need to enforce the arrow-head structure on matrix $X$. In other words, we need to translate the arrow-head structure requirement into linear constraints. Thus, we introduce symmetric matrices $\Check{\mathtt{A}}_{hl}$ for $2\leq h < l \leq n$ with $(\Check{\mathtt{A}}_{hl})_{hl} = (\Check{\mathtt{A}}_{hl})_{lh} = 1$ for $2\leq h<l\leq n$, and all other entries are zero. Moreover, we introduce matrix $\Hat{\mathtt{A}}_k$, with $(\Hat{\mathtt{A}}_k)_{11} = 1, (\Hat{\mathtt{A}}_k)_{kk} = -1$ for $k=2,...,n$, and rest of the entries are zero. Therefore, we write the following model
\begin{equation} \label{mo:2dsdop} \tag{$\Pcal^{P}_{SDO}$}
    \begin{aligned}
    \min        \   & \trace{(\Vec{\mathtt{C}} X)}\\
    \text{  s.t. }  & \trace{(\Vec{\mathtt{A}}_i  X)}  = b_i, && i = 1,...,m \\
                    & \trace{(\Check{\mathtt{A}}_{hl} X)}  =  0 ,  && h = 2, ..., n-1, \ h < l \leq n\\
                    & \trace{(\Hat{\mathtt{A}}_k    X)}  =  0 , && k = 2,...,n \\
                    & \ X\succeq 0, 
    \end{aligned}
\end{equation}
which is an accurate SDO representation of \eqref{mo:socop} as the arrow-head structure of the matrix $X$ is enforced by the linear constraints. Let $z_{\Pcal^{P}_{SDO}}^*$ denote the optimal objective function value of \eqref{mo:2dsdop}, and 
\begin{align*}
    \Fcal_{\Pcal^{P}_{SDO}} = \{ X \in \Smbb^n : &\trace{(\Vec{\mathtt{A}}_i  X)}  = b_i, i = 1,...,m,\\ &\trace{(\Check{\mathtt{A}}_{hl} X)}  =  0, h = 2, ..., n, \ h < l,\\  &\trace{(\Hat{\mathtt{A}}_k    X)}  =  0, k = 2,...,n,\\
    & \ X \succeq 0\},
\end{align*}
and
\begin{align*}
    {\Pcal^{P}_{SDO}}^* &= \{ X \in \Fcal_{\Pcal^{P}_{SDO}} : \trace{(\Vec{\mathtt{C}} X)} = z_{\Pcal^{P}_{SDO}}^*\},
\end{align*}
be the feasible and optimal solution sets of problem \eqref{mo:2dsdop}, respectively. 
Next, we present the dual model of the SDO problem \eqref{mo:2dsdop} as
\begin{equation} \label{mo:2dsdod} \tag{$\Dcal^{P}_{SDO}$}
    \begin{aligned}  
    \max   \        &     b^T v\\
    \text{  s.t.  } & \sum_{i=1}^m v_i \ \Vec{\mathtt{A}}_i + \sum_{h\neq 1, h < l} w_{hl} \Check{\mathtt{A}}_{hl} + \sum_{k=2}^n u_k \Hat{\mathtt{A}}_k + S = \Vec{\mathtt{C}},\\
                    & \ S\succeq0,
    \end{aligned}
\end{equation}
where $w_{hl}$ denotes the dual variable corresponding to the matrix dedicated for setting entries $(h,l)$ and $(l,h)$ in $X$ equal to zero, and $u_k$ corresponds to the linear constraints that are setting elements on the diagonal of each block equal to each other for that specific block. \\
Similarly, let $z_{\Dcal^{P}_{SDO}}^*$ denote the optimal objective function value for the dual model \eqref{mo:2dsdod}, and 
\small
\begin{align*}
    \Fcal_{\Dcal^{P}_{SDO}} = \bigg\{&(v,w,u,S) \in \Rmbb^m \times \Rmbb^{\frac{(n-1)(n-2)}{2}} \times \Rmbb^{(n-1)} \times \Smbb^n :
    \sum_{i=1}^m v_i \ \Vec{\mathtt{A}}_i + \sum_{h\neq 1, h < l} w_{hl} \Check{\mathtt{A}}_{hl} + \sum_{k=2}^n u_k \Hat{\mathtt{A}}_k + S = \Vec{\mathtt{C}}, \  S \succeq 0 \bigg\}
\end{align*}
\normalsize
and
\begin{align*}
    {\Dcal^{P}_{SDO}}^* &= \{(v,w,u,S) \in \Fcal_{\Dcal^{P}_{SDO}}: b^T v = z_{\Dcal^{P}_{SDO}}^*\},
\end{align*}
be the feasible and optimal solution sets of problem \eqref{mo:2dsdod}, respectively. 
% Here, we can rewrite the dual constraint as
% \begin{align*}
%      S &= \Cvec - \big[\sum_{i=1}^m v_i \Avec_i +\sum_{h\neq 1, h < l} w_{hl} \Atilde_{hl} + \sum_{k=2}^n u_k \Ahat_k \big].
% \end{align*}
% Here, for brevity, let
% \begin{equation} \label{eq:sig}
%     \sigma_j = c_j - \sum_{i=1}^{m}y_i a_{ij} \qquad \text{ for all } j = 1,...,n. 
% \end{equation}
% Then, \textcolor{red}{[tiny font size]}
% \tiny
% \begin{align*}
%     S &= 
%     \begin{bmatrix}
%     \frac{(c_1 - \sum_{i=1}^{m}y_i a_{i1})}{n} - \sum_{k=2}^n u_k      &  \frac{(c_2 - \sum_{i=1}^{m}y_i a_{i2})}{2}      &   \hdots                    & \hdots                  & \frac{(c_n - \sum_{i=1}^{m}y_i a_{in})}{2}\\
%     \frac{(c_2 - \sum_{i=1}^{m}y_i a_{i2})}{2}                         &  \frac{(c_1 - \sum_{i=1}^{m}y_i a_{i1})}{n}+u_2  &  -w_{23}           & \hdots                  &  -w_{2n} \\
%     \vdots                                        &  -w_{23}           &  \frac{(c_1 - \sum_{i=1}^{m}y_i a_{i1})}{n}+u_3  & \ddots                  &  \vdots   \\
%     \vdots                                        &   \vdots                    &   \ddots                    & \ddots                  &  -w_{(n-1)n} \\ 
%     \frac{(c_n - \sum_{i=1}^{m}y_i a_{in})}{2}                         &  -w_{2n}           &   \hdots                    & -w_{(n-1)n} & \frac{(c_1 - \sum_{i=1}^{m}y_i a_{i1})}{n}+u_n \\
%     \end{bmatrix}.
% \end{align*}
% \normalsize

Let's analyze the structure of $S$. First, recall problem \eqref{mo:socod}, where we have
\begin{equation*}
    \sbar_j = {c}_{j} - \sum_{i=1}^m \ybar_i {a_{ij}}, \qquad \text{ for all } j = 1,...,n.
\end{equation*}
Considering that, we seek to preserve the objective function value throughout the mapping, we have $v = \ybar$. Therefore, using the non-zero structure of the matrices $\Check{\mathtt{A}}_{hl}$ and $\Hat{\mathtt{A}}_k$, we have
\begin{align} \label{eq:derivedS}
    S &= 
    \begin{bmatrix}
    \frac{1}{n}(\sbar_1) - \sum_{k=2}^n u_k      &  \frac{1}{2}(\sbar_2)      &   \hdots                    & \hdots                  & \frac{1}{2}(\sbar_n)\\
    \frac{1}{2}(\sbar_2)                         &  \frac{1}{n}(\sbar_1)+u_2  &  -w_{23}           & \hdots                  &  -w_{2n} \\
    \vdots                                        &  -w_{23}           &  \frac{1}{n}(\sbar_1)+u_3  & \ddots                  &  \vdots   \\
    \vdots                                        &   \vdots                    &   \ddots                    & \ddots                  &  -w_{(n-1)n} \\ 
    \frac{1}{2}(\sbar_n)                         &  -w_{2n}           &   \hdots                    & -w_{(n-1)n} & \frac{1}{n}(\sbar_1)+u_n \\
    \end{bmatrix}.
\end{align}
Furthermore, we have free variables $w$ and $u$ which can take values such that $S$ is positive semidefinite. Now, let $y = (v;w;u)$, and $\bbar = [b;\textbf{0}_{\frac{(n-1)(n-2)}{2} \times 1}; \textbf{0}_{n-1 \times 1}]$. Then, we can present the following theorem which presents a point to set admissible mapping, see Definition \ref{def:admismap}, with $r=1$, based on \eqref{mo:socop} and \eqref{mo:socod}, and their representations \eqref{mo:2dsdop} and \eqref{mo:2dsdod}. 

% \begin{definition}
% A mapping $\Mcal$ is called admissible if it preserves feasibility and objective function value, i.e.
% \begin{align*}
%     (\xbar,\ybar, \sbar) \in \Fcal_{\Bar{\Pcal}}\times \Fcal_{\Bar{\Dcal}}  &\Rightarrow \Mcal(\xbar,\ybar, \sbar) \in \Fcal_{\Pcal''} \times \Fcal_{\Dcal''},\\
%     (X,y,w,u,S) \in \Fcal_{\Pcal''} \times \Fcal_{\Dcal''} &\Rightarrow  \Mcal^{-1}(X,y,w,u,S) \in \Fcal_{\Bar{\Pcal}}\times \Fcal_{\Bar{\Dcal}},\\
%     c^T \xbar = \trace{(\Cvec X)}&, \  
%     b^T \ybar = b^T y.  
% \end{align*}
% \end{definition}
% Based on this definition, we can present the following theorem.
\begin{theorem} \label{th:6}
Consider the SOCO problem pairs (\ref{mo:socop}) and (\ref{mo:socod}) with $(\xbar, \ybar, \sbar) \in \Fcal_{\Pcal^1_{SOCO}} \times \Fcal_{\Dcal^1_{SOCO}}$ and SDO problem pairs (\ref{mo:2dsdop}) and (\ref{mo:2dsdod}) with $(X,v,w,u,S) \in \Fcal_{\Pcal^{P}_{SDO}} \times \Fcal_{\Dcal^{P}_{SDO}}$. Then, the mapping $(X,v,w,u,S) = \Mcal(\xbar,\ybar,\sbar)$ defined as
\small
\begin{align*}
    X &= \arw(\xbar), \\
    v &= \ybar, \\
    S &=     
    \begin{bmatrix}
    \frac{\sbar_1}{n} - \sum_{k=2}^n u_k      &  \frac{\sbar_2}{2}     &   \hdots                    & \hdots                  & \frac{\sbar_n}{2}\\
    \frac{\sbar_2}{2}                       &  \frac{\sbar_1}{n}+u_2  &  -w_{23}           & \hdots                  &  -w_{2n} \\
    \vdots                                        &  -w_{23}           &  \frac{\sbar_1}{n}+u_3  & \ddots                  &  \vdots   \\
    \vdots                                        &   \vdots                    &   \ddots                    & \ddots                  &  -w_{(n-1)n} \\ 
    \frac{\sbar_n}{2}                        &  -w_{2n}           &   \hdots                    & -w_{(n-1)n} & \frac{\sbar_1}{n}+u_n \\
    \end{bmatrix},
\end{align*}
\normalsize
where $w \in \Rmbb^{\frac{(n-1)(n-2)}{2}}$ and $u \in \Rmbb^{(n-1)}$ taking any values such that $S \succeq 0$ is a point-to-set admissible mapping. The inverse mapping is denoted by $(\xbar,\ybar,\sbar) = \Mcal^{-1}(X,v,w,u,S)$, with
\begin{align*}
    \xbar = \arw^{-1}(X), \qquad
    \ybar = v, \qquad
    \sbar = \begin{bmatrix}
    c_1 - \sum_{i=1}^{m} v_i a_{i1}\\
    \vdots \\
    c_n - \sum_{i=1}^{m} v_i a_{in}
    \end{bmatrix}.
\end{align*}
\end{theorem}
\begin{proof}
The proof of this theorem is presented in Appendix \ref{appsec:3}.
\end{proof}
\begin{remark}
The proposed mapping in Theorem \ref{th:6} is not necessarily a point to point mapping regarding dual variables of \eqref{mo:socod} and \eqref{mo:2dsdod}.
\end{remark}
Analogous to Corollaries \ref{col:d-obj}-\ref{col:d-comp}, the proposed mapping in Theorem \ref{th:6} preserves optimality and complementarity.
\subsection{Rank-one Mapping}
Using the intuition from the dual side's section, we seek to represent matrix $S$ using a rank one matrix. Hence, we propose
\begin{align*}
    S = \eta (\eta)^T,
    % = 
    % \begin{bmatrix}
    % (\beta^{'}_{1})^2       & \beta^{'}_{1} \beta^{'}_{2} & \hdots & \beta^{'}_{1} \beta^{'}_{n} \\
    % \beta^{'}_{1} \beta^{'}_{2} & (\beta^{'}_{2})^2       & \hdots & \beta^{'}_{2} \beta^{'}_{n} \\
    % \vdots          & \vdots          & \ddots & \vdots \\
    % \beta^{'}_{1} \beta^{'}_{n} & \beta^{'}_{2} \beta^{'}_{n} & \hdots & (\beta^{'}_{n})^2
    % \end{bmatrix},
\end{align*}
with
\begin{align*}
    \sum_{i=1}^n (\eta_{i})^2 &=  \sbar_1,\\
    \eta_1 \eta_j &= \frac{\sbar_j}{2} \text{  for all } j = 2, ..., n.
\end{align*}
We need to solve an $n$-variable-$n$-equation system. The solution of this system is
\begin{align*}
    \eta &= \frac{1}{\sqrt{2(\sbar_1 + \delta')}}\big(\sbar_1 + \delta', \sbar_2, ..., \sbar_n \big)^T,
\end{align*}
where $\delta' = \sqrt{\sbar_1 - ||\sbar_{2:n}||^2}$. Using this vector, we can construct matrix $S$. Thus, we can compute $u_k$ for $k=2,...,n$, and $w_{hl}$ for $h\neq 1, h<l$ as follows,\\ 
\begin{align*}
    u_k &= (\eta_k)^2 - \frac{1}{n} \sbar_1
    = \frac{\sbar_k^2}{2(\sbar_1 + \delta')} - \frac{1}{n} \sbar_1, && k = 2,...,n,\\
    w_{hl} &= \eta_{h}\eta_{l}
    = \frac{\sbar_h \sbar_l}{2(\sbar_1 + \delta')}, && h \neq 1, h<l.
\end{align*}

\begin{theorem} \label{th:7} 
Consider the rank-one mapping
\begin{align} \label{mo:r1maps}
    S =
    \begin{cases}
    [0]_{n \times n} & \text{ if } \sbar = (0,0,...,0),\\
    {\rm PMR1}              & \text{ otherwise}.
    \end{cases}
\end{align}
where
\begin{equation*}
    {\rm PMR1} = \eta (\eta)^T =  
    \begin{bmatrix}
    \frac{\sbar_1 + \delta'}{2} & \frac{\sbar_2}{2} & \hdots & \frac{\sbar_n}{2} \\
    \frac{\sbar_2}{2} & \frac{\sbar_2^2}{2[\sbar_1 + \delta']} & \hdots & \frac{\sbar_2 \sbar_n}{2[\sbar_1 + \delta']} \\
    \vdots    & \vdots    & \ddots & \vdots    \\
    \frac{\sbar_n}{2} & \frac{\sbar_2 \sbar_n}{2[\sbar_1 + \delta']} & \hdots & \frac{\sbar_n^2}{2[\sbar_1 + \delta']} \\
    \end{bmatrix}.
\end{equation*}
Then, (\ref{mo:r1maps}) together with $(X,y) = ({\rm Arw}(\xbar),\ybar)$ is a point-to-point admissible mapping.
\end{theorem}
\begin{proof}
The proof is similar to that of Theorem \ref{th:2}.
\end{proof}
\subsection{Higher Rank Mapping}
The derivation of full rank mapping is similar to the derivation discussed in Section \ref{ssec:FRM}, as we show that when a SOCO solution $\sbar \in {\rm int}(\Lcal^n)$, then we can use a full rank mapping, i.e.
\begin{equation*}
    {\rm PMR} =\sum_{i=1}^{n}\eta^i(\eta^i)^T.
\end{equation*}
The proof of existence of a full rank mapping is similar to the proof presented in Appendix \ref{appsec:2}.

\begin{theorem} \label{th:8}
There exist mappings  ${\rm PMR}$ where {\rm rank}$({\rm PMR}(\bar{s}))=n$ for $\sbar\in {\rm int}(\Lcal^n)$.
\end{theorem}
\begin{proof}
The proof of this theorem is similar to the proof presented in Appendix \ref{appsec:2}. The difference is that we need to set $\pi^1 = \sbar$ in the proof. 
\end{proof}
The reason why the procedure to show the existence of full rank mapping is similar for both directions is that although the mappings seem to be different, in fact both satisfy analogous conditions about the first row and column and the trace of the mapping matrix. From this point of view, they are very similar. In fact, the update vector in the procedure is meant to be satisfying those conditions. Thus, we can follow the procedure in Appendix \ref{appsec:2} for both mappings. 

%%%%%%%%%%%%%%%%%%%%%%%%%%%%
Similar to Theorem 1 of \cite{sim2005note}, we can develop the following mapping for the primal side.
$${\rm MR}(\sbar)=\begin{pmatrix}
\frac{1}{4}\theta'&\frac{1}{2}\sbar_{2:n}^T\\
\frac{1}{2}\sbar_{2:n}&\frac{\sbar_1-\|\sbar_{2:n}\|}{2(n-1)}I+\frac{\sbar_{2:n}\sbar_{2:n}^T}{\theta'}
\end{pmatrix},$$
where $\theta'=\sbar_1+\|\sbar_{2:n}\|+\sqrt{(\sbar_1+\|\sbar_{2:n}\|)^2-4\|\sbar_{2:n}\|^2}$. Here, we can perform the same analysis we did in Section \ref{ssec:FRM}, too. Considering the case in which the solution is on the boundary of the second-order cone, i.e. $\sbar_1 = \|\sbar_{2:n}\|$. Then, $\theta' = 2\sbar_1$, and 
$${\rm MR}^1 (\sbar)=\begin{pmatrix}
\frac{1}{2}\sbar_1&\frac{1}{2}\sbar_{2:n}^T\\
\frac{1}{2}\sbar_{2:n}&\frac{\sbar_{2:n}\sbar_{2:n}^T}{2\sbar_1}
\end{pmatrix}.$$
We can see that this is exactly identical to the rank one mapping we presented in Theorem \ref{th:7}. We can also write it as 
\begin{equation*}
    {\rm MR}^1 (\sbar)= \gamma^1 (\gamma^1)^T
\end{equation*}
where 
\begin{equation*}
    \gamma^1 = \frac{1}{\sqrt{\theta'}} \left( \frac{1}{2} \theta', \sbar_{2}, \dots, \sbar_{n} \right)^T.
\end{equation*}
To obtain ${\rm MR}(\sbar)$ as the sum of rank-one matrices, we define
\begin{equation*}
    \gamma^j = \sqrt{\frac{\sbar_1 - \|\sbar_{2:n} \|}{2(n-1)}} \ e_j \qquad {\rm for } j =2,\dots,n,
\end{equation*}
where $e_j$ is a unit vector with 1 in element $j$. By this setting, we have 
$$
{\rm MR}(\sbar)= \sum_{j=1}^{n}\gamma^j(\gamma^j)^T.
$$
If the solution $\sbar$ of the SOCO dual problem is on the boundary of the cone, i.e. $\sbar_1 = \|\sbar_{2:n} \|$, then we get
\begin{align*}
    \gamma^1 &= \frac{1}{\sqrt{2\sbar_1}} \left(\sbar_1, \sbar_{2}, \dots, \sbar_{n} \right)^T, \\
    \gamma^j &= 0, \qquad \qquad \text{  for } j = 2,\dots, n.
\end{align*}
This leads to the rank one mapping ${\rm MR}^1 (\sbar)$. On the other hand, if the solution $\sbar \in {\rm int}(\Lcal^n)$, i.e. $\sbar_1 > \|\sbar_{2:n} \|$, then $\gamma^j \neq 0$ for $j = 2,\dots, n$, and we obtain the full rank mapping ${\rm MR}(\sbar)$. To construct an admissible rank-$k$ mapping, here one cannot take a combination of first $k$ vectors $\gamma^j$ as it violates the conditions given in Theorem \ref{th:6}, i.e. the sum of the diagonal elements will not be equal to $\sbar_1$. Recall the definition of $\Ncal$ from Section \ref{ssec:FRM}. The correct choice of $\gamma^j$ to construct a rank-$k$ mapping is
\begin{equation*}
    \gamma^j = \sqrt{\frac{\sbar_1 - \|\sbar_{2:n} \|}{2(k-1)}} \ e_j \qquad {\rm for } \ \ j \in \Ncal, \ {\rm and } \ \gamma^j = 0 \ {\rm for } \ \ j \notin \Ncal.
\end{equation*}
The resulting matrix ${\rm MR}(\sbar)$ has rank $k$, and one can easily see that it satisfies the conditions of Theorem \ref{th:6}. Analogous to the dual side, we have the following theorem.
\begin{theorem}\label{th:9}
Let $\rho(\sbar)=\max\{\rank (\Mcal(\sbar)): \Mcal \text{ is admissible map} \}$. We have 
\begin{itemize}
    \item $\rho(\sbar)=n$ if $\sbar\in {\rm int}(\Lcal^n).$
    \item $\rho(\sbar)=1$ if $\sbar\in \partial(\Lcal^n).$
\end{itemize}
\end{theorem}
\begin{proof}
Proof of this theorem is similar to Theorem 1 of \cite{sim2005note}.
\end{proof}
\subsection{Generalization to Multiple SOCs} \label{ss:psgen}
Similar to the discussion in Subsection \ref{ss:dsgen}, here we adopt the mapping presented in Theorem \ref{th:5} to the primal side. Here we need to define the following notations. Let,
\begin{align*}
    \Tilde{\mathtt{C}} =     
    \begin{bmatrix}
    \Vec{\mathtt{C}}^1 &         &        &                     \\
            & \Vec{\mathtt{C}}^2   &        &                     \\
            &           & \ddots &                     \\
            &           &        & \Vec{\mathtt{C}}^r
    \end{bmatrix},
    \quad
    \text{ where }
    \quad
    \Vec{\mathtt{C}}^k &= 
    % \begin{bmatrix}
    % \frac{1}{n} c_{1}^{k}    & \frac{1}{2} c_{2}^{k} &  ...   & \frac{1}{2} c_{n}^{k} \\
    % \frac{1}{2} c_{2}^{k}    & \frac{1}{n} c_{1}^{k} &        &     \\
    % \vdots                   &                       & \ddots &    \\
    % \frac{1}{2} c_{n}^{k}    &                       &        & \frac{1}{n} c_{1}^{k} 
    % \end{bmatrix}\\
    % &= 
    \arw(\frac{1}{n} c_{1}^{k}, \frac{1}{2} c_{2}^{k},  ..., \frac{1}{2} c_{n}^{k})
\end{align*}
and
\begin{align*}
    \Tilde{\mathtt{A}}_j = \text{DArw}(\Vec{\mathtt{A}}_j^1, \Vec{\mathtt{A}}_j^2, \dots ,\Vec{\mathtt{A}}_j^r) 
    % = 
    % \begin{bmatrix}
    % \Avec_j^1 &             &        &             \\
    %             & \Avec_j^2 &        &             \\
    %             &             & \ddots &             \\
    %             &             &        & \Avec_j^r
    % \end{bmatrix} \qquad 
    \qquad \text{ for all } j= 1,...,m,
\end{align*}
where 
\begin{align*}
    \Vec{\mathtt{A}}_j^{k}  = 
    % &= 
    % \begin{bmatrix}
    % \frac{1}{n} a_{i1}^{k}    & \frac{1}{2} a_{i2}^{k} &  ...   & \frac{1}{2} a_{in}^{k} \\
    % \frac{1}{2} a_{i2}^{k}    & \frac{1}{n} a_{i1}^{k} &        &     \\
    % \vdots                    &                        & \ddots &    \\
    % \frac{1}{2} a_{in}^{k}    &                        &        & \frac{1}{n} a_{i1}^{k} 
    % \end{bmatrix}\\
    % &= 
    \arw(\frac{1}{n} a_{i1}^{k}, \frac{1}{2} a_{i2}^{k}, ..., \frac{1}{2} a_{in}^{k})
    \qquad \text{ for all } k= 1,...,r.
\end{align*}

Similar to the structure of $ \Tilde{\mathtt{C}}$ and $\Tilde{\mathtt{A}}_j$, matrix $\Xtilde$ needs to have a block-diagonal structure. Thus, we need to enforce this structure on it. To this end, we define an operator. Let ${\rm Cf}(i) = j$, be an operator which returns the corresponding cone $j$ to input row $i$ of the block-diagonal matrices. %This operator helps us to define sets which includes elements of matrix $\Xtilde$ which need to have a certain value to obtain the arrow-head matrices in block-diagonal structure. 
%Using this operator we define sets $\Ical$ and $\Kcal$.
Let set $\Ical$ consists of all entries that need to be zero due to being either off-arrow within the blocks, or the off-block-diagonal structure.
% \textcolor{blue}{
% For a given $i$, for off-arrow
% row and columns within
% \begin{equation*}
%     [\sum_{t=1}^{{\rm Cf}(h)-1} n_t +1, \sum_{t=1}^{{\rm Cf}(h)} n_t]
% \end{equation*}\\
% \begin{equation*}
%     h_i \in [\sum_{t=1}^{{\rm Cf}(h)-1} n_t +1, \sum_{t=1}^{{\rm Cf}(h)} n_t], \qquad h_i < l_i \leq \sum_{t=1}^{{\rm Cf}(h)} n_t.
% \end{equation*}\\
% For off-diagonal
% \begin{equation*}
%     \sum_{t=1}^{{\rm Cf}(h)} n_t < l_i
% \end{equation*}\\
% For arrow-head diagonal
% \begin{equation*}
%     h_1 = \sum_{t=1}^{i-1} n_t +1, \qquad h_1 < j \leq \sum_{t=1}^{i} n_t.
% \end{equation*}
% }
\begin{align*}
    \Ical &= \Bigg\{(h,l) \in [1, \sum_{i=1}^{r} n_i]^2 \Bigg| 
    \begin{cases}
    \sum_{i=0}^{{\rm Cf}(h)} n_i < l, \\
    \sum_{i=0}^{{\rm Cf}(h)-1} n_i +1 < h < \sum_{i=0}^{{\rm Cf}(h)} n_i, \quad {\rm and} \quad h < l \leq \sum_{i=0}^{{\rm Cf}(h)} n_i.
    \end{cases}\Bigg\},\\
    \Kcal &= \Bigg\{ k \in [1, \sum_{i=1}^{r} n_i] \Bigg| \sum_{i=0}^{{\rm Cf}(k)-1} n_i +1 < k \leq \sum_{i=0}^{{\rm Cf}(k)} n_i \Bigg\},
\end{align*}
where $n_0 = 0$. Note that here we introduce matrices $\Tilde{\Check{\mathtt{A}}}_{hl}$ and $\Tilde{\Hat{\mathtt{A}}}_k$ which are generalizations of $\Check{\mathtt{A}}_{hl}$ and $\Hat{\mathtt{A}}_k$, for the multiple cone case, respectively. In detail, $\Tilde{\Check{\mathtt{A}}}_{hl}$ enforces all entries off-block-diagonal and off-arrow within each block to be zero. Moreover, $\Tilde{\Hat{\mathtt{A}}}_k$ guarantees that in each block, diagonal entries are equal by setting entry $(k,k)$ equal to $-1$ and entry $(\sum_{t=1}^{{\rm Cf}(k)-1} n_t +1, \sum_{t=1}^{{\rm Cf}(k)-1} n_t +1)$ equal to $+1$.
By this notation, we have
\begin{equation} 
\label{mo:genprimalp} \tag{$\Pcal_{SDO}^{\Ptilde}$}
    \begin{aligned}
    \min        \   & \trace{(\Tilde{\mathtt{C}} \Xtilde)}\\
    \text{  s.t. }  & \trace{(\Tilde{\mathtt{A}}_j  \Xtilde)}  = b_j, && j = 1,...,m \\
                    & \trace{(\Tilde{\Check{\mathtt{A}}}_{hl} \Xtilde)}  =  0 ,  && (h,l) \in \Ical \\
                    & \trace{(\Tilde{\Hat{\mathtt{A}}}_k \Xtilde)}  =  0 , && k \in \Kcal \\
                    & \Xtilde \succeq 0. 
    \end{aligned}
\end{equation}
Then, we dualize and get
\begin{equation} 
\label{mo:genprimald} \tag{$\Dcal_{SDO}^{\Ptilde}$}
    \begin{aligned}
    \max        \   & b^T y\\
    \text{  s.t. }  & \sum_{j=1}^m y_j \Tilde{\mathtt{A}}_j + \sum_{h,l \in \Ical} v_{hl} \Tilde{\Check{\mathtt{A}}}_{hl} + \sum_{k \in \Kcal} z_k \Tilde{\Hat{\mathtt{A}}}_k + \Stilde = \Tilde{\mathtt{C}},\\
                    & \Stilde \succeq 0.
    \end{aligned}
\end{equation}
In similar fashion to other introduced models, we can define the sets of feasible and optimal solutions corresponding to models \eqref{mo:genprimalp} and \eqref{mo:genprimald}.

Next, we can present the following theorem. 
\begin{theorem} \label{th:10}
Consider the SOCO problem pairs (\ref{mo:gensocop}) and (\ref{mo:gensocod}) with 
\begin{equation*}
    (\xbar^1;\xbar^2;...;\xbar^r;\ybar;\sbar^1;\sbar^2;...;\sbar^r) \in \Fcal_{\Pcal_{SOCO}} \times \Fcal_{\Dcal_{SOCO}},
\end{equation*}
and SDO problem pairs \eqref{mo:genprimalp} and \eqref{mo:genprimald} with $(\Xtilde,\ytilde,\Stilde) \in \Fcal_{\Pcal_{SDO}^{\Ptilde}} \times \Fcal_{\Dcal_{SDO}^{\Ptilde}} $. Then, the mapping $(\Xtilde,\ytilde,v,z,\Stilde) = \Mcal(\xbar,\ybar,\sbar)$ defined as
\begin{align*}
    \Xtilde &= 
    {\rm DArw}(X^1, X^2, ..., X^r),\\
    \ytilde &= \ybar, \\
    \Stilde &= 
    \tiny
    \begin{bmatrix}
    S^1 &     &        & \\
        & S^2 &        & \\
        &     & \ddots & \\
        &     &        & S^r
    \end{bmatrix}
    \normalsize
    \succeq 0  \ {\rm s.t.} 
    \begin{bmatrix}
    \sum_{j=u_i}^{u_i+n_i} \Stilde_{jj}  \\
    \Stilde_{u_i,u_i+1}                  \\
    \vdots                               \\
    \Stilde_{u_i,u_i+n_i}            
    \end{bmatrix} 
    = 
    \begin{bmatrix}
    \sbar_1^i\\
    \frac{\sbar_2^i}{2}\\
    \vdots\\
    \frac{\sbar_n^i}{2}
    \end{bmatrix}
\end{align*}
with vectors $v$ and $z$ taking values such that $\Stilde \succeq 0$. Furthermore, the inverse mapping denoted by $(\xbar,\ybar,\sbar) = \Mcal^{-1}(\Xtilde,\ytilde,v,z,\Stilde)$, with
\begin{align*}
    (\xbar^1;...;\xbar^r) \text{ with } \xbar^i =  
    \begin{bmatrix}
    \xtilde_{u_i,u_i} \\
    \xtilde_{u_i +1 ,u_i} \\
    \vdots     \\
    \xtilde_{u_i +n_i ,u_i}
    \end{bmatrix},
    \quad
    \ybar = \ytilde, \quad
    (\sbar^1;...;\sbar^r) \text{ with } \sbar^i =  
    \begin{bmatrix}
    \sum_{j=u_i}^{u_i+n_i} \Stilde_{jj}  \\
    2\Stilde_{u_i,u_i+1}                 \\
    \vdots                               \\
    2\Stilde_{u_i,u_i +n_i}            
    \end{bmatrix},
\end{align*}
where $u_i = 1+\sum_{k=0}^{i-1} n_{i}$ and $n_0 = 0$, is an admissible mapping.
\end{theorem}
\begin{proof}
The proof of this theorem is analogous to that of Theorem \ref{th:1}.
\end{proof}
We observe that although in the general case the mapping is similar to that of the dual side, but the details of the mapping are different. This mapping is different from the other one since it requires more work to enforce the arrow-head structure on matrix $\Xtilde$. An analogous analysis on different rank mappings can be done for the generalized multiple cone case, with the difference that here we have the arrow-head structure for matrix $\Xtilde$, and we can compute matrix $\Stilde$ with different ranks. We skip that similar analysis here for brevity.

Now, that we have the generalized standard mapping for both sides, we can proceed with studying the mapping of the optimal partitions on both sides in the next section.
% Exploiting notations of $\Ctilde$ and $\Atilde_j$, we can propose a generalized rank-one mapping as follows,
% \begin{align*}
%     \textcolor{red}{\Xtilde} &= ,\\ 
%     \ytilde &= \ybar,
% \end{align*}
% and for matrix $\Stilde$, we have
% \begin{equation*}
%     \Stilde = \bar{\beta} \bar{\beta}^T,
% \end{equation*}
% where
% \begin{align*}
%     \bar{\beta} = (\bar{\beta}^1_1, \bar{\beta}^1_2, ..., \bar{\beta}^1_{n_1}, \bar{\beta}^2_1, \bar{\beta}^2_2, ..., \bar{\beta}^2_{n_2}, ..., \bar{\beta}^r_1, \bar{\beta}^r_2, ..., \bar{\beta}^r_{n_r})^T.
% \end{align*}
% One can calculate vector $\beta'$ in the closed form as follows,
% \begin{align*}
%     \bar{\beta}^i_j = 
%     \begin{cases} 
%     \sqrt{\frac{\sbar^i_1 + {\delta'}^i}{2}} & j=1,\\
%     \frac{\sbar^i_j}{2\sqrt{\frac{\sbar^i_1 + {\delta'}^i}{2}}} & j>1,
%     \end{cases}
% \end{align*}
% where ${\delta'}^i = \sqrt{(\sbar^i_1)^2 - ||\sbar^i_{2:n}||^2}$.

\section{Mapping the Optimal Partition} \label{sec:5}
As shown in the previous sections, the proposed mappings represent a solution of SOCO depending on where it is located in the cone. For a solution on the boundary of the second-order cone, all admissible maps provide a rank-one positive semidefinite matrix, while a solution in the interior of the Lorentz cone can be mapped to  semidefinite matrices with different ranks. This correspondence can be analyzed in order to see how the optimal partition of SOCO is mapped to that of the derived SDO counterparts. For mapping the optimal partition, maximally complementary solutions are of particular interest. First, we define a helpful notation, and then the following theorem discusses the preservation of maximal complementarity. 
\begin{definition}[Proper Map]
Mapping $\Mcal$ is a \textit{proper map} if $\Mcal$ is admissible and ${\rm rank}(X)=\rho(\xbar)$ for all $\xbar\in \Lcal^n$, or ${\rm rank}(S)=\rho(\sbar)$ for all $\sbar\in \Lcal^n$
\end{definition}
Based on this definition, a rank-one mapping is not proper but map ${\rm MR}$ of \cite{sim2005note} is proper.
In the subsequent theorem, we show that a proper mapping preserves maximal complementarity using eigenvalues of the mapped solution. We used the mapping approach of Section~\ref{sec:3} which starts from dual side. 
Let $\lambda_i^X$ denotes the $i^{th}$ eigenvalue of a matrix $X$. Then, the eigenvalues of an arrow-head  matrix $X=\arw(\xbar)$ are $\lambda_1^{X}=\xbar_1-\|\xbar_{2:n}\|$, $\lambda_2^{X}=\dots=\lambda_{n-1}^{X}=\xbar_1$, and $\lambda_{n}^{X}=\xbar_1+\|\xbar_{2:n}\|$, see e.g. \cite{alizadeh2003second}.

\begin{theorem} \label{th:11}
For a maximally complementary solution $(\xbar; \ybar; \sbar) \in \Pcal^*_{SOCO} \times \Dcal^*_{SOCO}$, the mapped solution $(\Xtilde, \ytilde, \Stilde)=\Mcal(\xbar; \ybar; \sbar) \in {\Pcal^{\Dtilde^*}_{SDO}} \times {\Dcal^{\Dtilde^*}_{SDO}}$ is maximally complementary if $\Mcal$ is a proper map.
\end{theorem}
\begin{proof}
Given a maximally complementary solution $(\xbar; \ybar; \sbar) \in \Pcal^*_{SOCO} \times \Dcal^*_{SOCO}$, one can analyze the result of mapping based how the cones are partitioned. The following analysis holds for the dual side direction. In this direction, we have matrix $\Stilde$ as a block diagonal of arrow-head matrices. Since $\Mcal$ is proper map, the eigenvalues of matrix $\Xtilde$ can be partitioned based on Lorentz cones.  If vector $\xbar^i$ is on the non-zero boundary of Lorentz cone, the all corresponding eigenvalues are non-zero except the first one. If the vector $\xbar^i$ is in the interior of Lorentz cone, the corresponding eigenvalues are positive. For the point of the second-order cone, all the corresponding eigenvalues are zero. Thus, we have
\begin{align*}
    \text{if } i \in \bar{\Bcal} &=
    \begin{cases}
    \xbar_1^i > \|\xbar_{2:n_i}^i\|  &\to \text{ full rank matrix $X^i$ with } \lambda_1^{X^i}, \lambda_2^{X^i}, ..., \lambda_{n_i}^{X^i} > 0,\\
    \sbar^i = 0  &\to \text{ zero $S^i$ matrix with } \lambda_1^{S^i} = ... = \lambda_{n_i}^{S^i} = 0,
    \end{cases}\\
    \text{if } i \in \bar{\Ncal} &=
    \begin{cases}
    \xbar^i = 0  &\to \text{ zero matrix $X^i$ with } \lambda_1^{X^i} = ... = \lambda_{n_i}^{X^i} = 0,\\
    \sbar_1^i > \|\sbar_{2:n_i}^i\|  &\to \text{ arrow-head matrix $S^i$ with } \lambda_1^{S^i}, \lambda_2^{S^i}, ..., \lambda_{n_i}^{S^i} > 0,
    \end{cases}\\
    \text{if } i \in \bar{\Rcal} &=
    \begin{cases}
    \xbar_1^i = \|\xbar_{2:{n_i}}^i\| >0   &\to \text{ rank-one matrix $X^i$ with } \lambda_1^{X^i} >0, \lambda_2^{X^i} = ... = \lambda_{n_i}^{X^i}  = 0\\
    \sbar_1^i = \|\sbar_{2:{n_i}}^i\| >0  &\to \text{ arrow-head $S^i$ matrix with } \lambda_1^{S^i} = 0, \lambda_2^{S^i}, ..., \lambda_{n_i}^{S^i} > 0,
    \end{cases}\\
    \text{if } i \in \bar{\Tcal}_1 &=
    \begin{cases}
    \xbar^i = 0 &\to \text{ zero matrix $X^i$ with } \lambda_1^{X^i} = ... = \lambda_{n_i}^{X^i} = 0,\\
    \sbar^i = 0 &\to \text{ zero matrix $S^i$ with } \lambda_1^{S^i} = ... = \lambda_{n_i}^{S^i} = 0,
    \end{cases}\\
    \text{if } i \in \bar{\Tcal}_2 &=
    \begin{cases}
    \xbar_1^i = \|\xbar_{2:{n_i}}^i\| > 0 &\to \text{ rank-one matrix $X^i$ with } \lambda_1^{X^i} >0, \lambda_2^{X^i} = ... = \lambda_{n_i}^{X^i}  = 0\\
    \sbar^i = 0  & \to \text{ zero matrix $S^i$ with } \lambda_1^{S^i} = ... = \lambda_{n_i}^{S^i} = 0,
    \end{cases}\\
    \text{if } i \in \bar{\Tcal}_3 &=
    \begin{cases}
    \xbar^i = 0 & \to \text{ zero matrix $X^i$ with } \lambda_1^{X^i} = ... = \lambda_{n_i}^{X^i} = 0,\\
    \sbar_1^i = \|\sbar_{2:{n_i}}^i\| > 0 & \to \text{ arrow-head matrix $S^i$ with } \lambda_1^{S^i} = 0, \lambda_2^{S^i}, ..., \lambda_{n_i}^{S^i} > 0,
    \end{cases}
\end{align*}
Now, using that $\Mcal$ is a proper, $\Mcal$ maps a given maximally complementary optimal SOCO solution  to $(\Xtilde, \ytilde, \Stilde) \in {\Pcal^{\Dtilde^*}_{SDO}} \times {\Dcal^{\Dtilde^*}_{SDO}}$. Here, we show that this solution is maximally complementary for the SDO counterpart problem. The proof goes by contradiction. Let us assume to the contrary that $(\Xtilde, \ytilde, \Stilde)$ is not a maximally complementary solution. Let $(\Xhat, \yhat, \Shat) \in {\Pcal^{\Dtilde^*}_{SDO}} \times {\Dcal^{\Dtilde^*}_{SDO}}$ be a maximally complementary solution. Since both $(\Xtilde, \ytilde, \Stilde)$ and $(\Xhat, \yhat, \Shat)$ are optimal, then  $\Shat \Xtilde = 0$. Accordingly, if $\lambda^{\Stilde^i}_j>0$ then $\lambda^{\Shat^i}_j>0$. Now it is enough to show that there are no $i,j$ such that $\lambda^{\Stilde^i}_j=0$ and $\lambda^{\Shat^i}_j>0$.  If there exist such index, by doing inverse mapping, the mapped solution in SOCO will have larger partition set either for $\bar{\Ncal}$, $\bar{\Rcal}$, or $\bar{\Tcal}_3$ which contradicts with assumption that solution $(\xbar; \ybar; \sbar )$ being a maximally complementary solution of SOCO. With similar reasoning we can show that there are no $i,j$ such that $\lambda^{\Xtilde^i}_j=0$ and $\lambda^{\Xhat^i}_j>0$.
Thus, we can conclude that a proper map preserves maximal complementarity. 
\end{proof}
One can prove similar theorem as stated as follows for the mapping of Section~\ref{sec:4} starting from primal side.
\begin{theorem} 
For a maximally complementary solution $(\xbar; \ybar; \sbar) \in \Pcal^*_{SOCO} \times \Dcal^*_{SOCO}$, the mapped solution $(\Xtilde, \ytilde, \Stilde)=\Mcal(\xbar; \ybar; \sbar) \in {\Pcal^{\Ptilde^*}_{SDO}} \times {\Dcal^{\Ptilde^*}_{SDO}}$ is maximally complementary if $\Mcal$ is a proper map.
\end{theorem}
\begin{proof}
    The proof is analogous to the proof of Theorem~\ref{th:11}.
\end{proof}

\subsection{Mapping the Optimal Partition}
% The effect of mapping on the optimal partition is defined based on how we represent a solution. 
Based on the analysis described in the proof of Theorem \ref{th:11}, one can develop the mapping for the optimal partitions. Based on the sign of the eigenvalues, if positive, their corresponding eigenvectors generate the subspaces $\Bcal$ and $\Ncal$, and if zero, to subspace $\Tcal$ of the optimal partition of SDO. First we consider mapping starting from the dual side as discussed in Section \ref{ss:dsgen} and Section \ref{ss:psgen}. 
The summary of optimal partition mapping on the dual side is presented in Table \ref{tab:opmds}, where  $e_j^i$ represents the $j^{{\rm th}}$ eigenvector corresponding to $\lambda_j$ for the semidefinite cone $i$.
\begin{table}[H]
    \centering
    \large
    \renewcommand{\arraystretch}{1.25}
    \begin{tabular}{|c|c|c|c|}
    \hline
                             & $\Bcal$                                              & $\Ncal$                              & $\Tcal$                        \\ \hline
    if $i \in \bar{\Bcal}$   & $\{ e^i_1, ... , e^i_{n_i} \}$                       & $\emptyset$                          & $\emptyset$                    \\ \hline
    if $i \in \bar{\Ncal}$   & $\emptyset$                                          & $\{ e^i_1, ... , e^i_{n_i} \}$       & $\emptyset$                    \\ \hline
    if $i \in \bar{\Rcal}$   & $\{ e^i_1\}$                                         & $\{ e^i_2, ... , e^i_{n_i} \}$       & $\emptyset$                    \\ \hline
    if $i \in \bar{\Tcal_1}$ & $\emptyset$                                          & $\emptyset$                          & $\{ e^i_1, ... , e^i_{n_i} \}$ \\ \hline
    if $i \in \bar{\Tcal_2}$ & $\{ e^i_1\}$                                         & $\emptyset$                          & $\{ e^i_2, ... , e^i_{n_i}\}$  \\ \hline
    if $i \in \bar{\Tcal_3}$ & $\emptyset$                                          & $\{ e^i_2, ... , e^i_{n_i} \}$       & $\{ e^i_1\}$                   \\ \hline
    \end{tabular}\\
    \caption{The Dual Side Optimal Partition Mapping}
    \label{tab:opmds}
\end{table}
A similar analysis can be conducted for the primal side derivation, where matrix $X$ has an arrow-head structure and we represent matrix $S$ using zero, rank-one or full rank matrices, based on the position of vector $\sbar$ in the second-order cone. The optimal partition mapping is presented in Table \ref{tab:opmps}.
\begin{table}[H]
    \centering
    \large
    \renewcommand{\arraystretch}{1.25}
    \begin{tabular}{|c|c|c|c|}
    \hline
                             & $\Bcal$                                              & $\Ncal$                              & $\Tcal$                        \\ \hline
    if $i \in \bar{\Bcal}$   & $\{ e^i_1, ... , e^i_{n_i} \}$                       & $\emptyset$                          & $\emptyset$                    \\ \hline
    if $i \in \bar{\Ncal}$   & $\emptyset$                                          & $\{ e^i_1, ... , e^i_{n_i} \}$       & $\emptyset$                    \\ \hline
    if $i \in \bar{\Rcal}$   & $\{ e^i_2, ... , e^i_{n_i} \}$                       & $\{ e^i_1\}$                         & $\emptyset$                   \\ \hline
    if $i \in \bar{\Tcal_1}$ & $\emptyset$                                          & $\emptyset$                          & $\{ e^i_1, ... , e^i_{n_i} \}$ \\ \hline
    if $i \in \bar{\Tcal_2}$ & $\{ e^i_2, ... , e^i_{n_i} \}$                       & $\emptyset$                          & $\{ e^i_1\}$                   \\ \hline
    if $i \in \bar{\Tcal_3}$ & $\emptyset$                                          & $\{e^i_1\}$                          & $\{e^i_2, ... , e^i_{n_i} \}$  \\ \hline
    \end{tabular}\\
    \caption{The Primal Side Optimal Partition Mapping}
    \label{tab:opmps}
\end{table}    
%%%%%%%%%%%%%%%%%%%%%
We can see that the result of mapping the optimal partition for the two sides are different. We observe identical outcome for partitions $\bar{\Bcal}$, $\bar{\Ncal}$, $\bar{\Tcal}_1$, which was predictable as they include the simple case where at least one of the variables $\xbar$ and $\sbar$ is zero. However, notable difference arises for the $\bar{\Rcal}$, $\bar{\Tcal_2}$, and $\bar{\Tcal_3}$ partitions depending on whether the dual or the primal variable's representation is forced to be arrow-head. 

Consider partition $\bar{\Rcal}$, where both $\xbar$ and $\sbar$ are on the non-zero boundary of the second-order cone. Adopting the dual (primal) side approach, variable $\sbar$ ($\xbar$) will have arrow-head representation, and the other has a rank one representation. We have seen in Example \ref{exm:1}, this partition is where   complementarity will not be preserved if we apply arrow-head representation on both primal and dual sides. While our derivations guarantee feasibility and duality, we can see that complementarity is also preserved as the arrow-head matrix has only one zero eigenvalue, while the rank-one matrix has only one none-zero. For partitions $\bar{\Tcal}_2$ and $\bar{\Tcal}_3$, the difference in tables simply comes from which variable is non-zero and its corresponding arrow-head representation eigenvalues.

From the geometric point of view, we can see that for partitions $\bar{\Bcal}$ and $\bar{\Ncal}$, where the respective solution is in the interior of the second-order cone, the mapped solution is in the interior of the semidefinite cone. For partition $\bar{\Rcal}$, we observe that the solution pair, which both are on the non-zero boundary of the second-order cone is mapped to a face of the semidefinite cone. For partition $\bar{\Tcal}_1$, the zero solution pair is mapped to the origin. Finally, for partitions $\bar{\Tcal}_2$, and $\bar{\Tcal}_3$, the solutions are on the boundary of the second-order cone in the primal and dual, respectively. We can see that the result of the mapping is a face of the semidefinite cone, for both cases, depending on if the primal or dual variable is on the boundary on the second-order cone.

%%%%%%%%%%%%%%%%%%%%%
\section{Conclusion} \label{sec:6}
In this paper, we study the relationship between a SOCO problem and its SDO representation. Knowing about the fact that SOCO can be considered as a special case of SDO, we extend the literature by investigating both the primal and the dual side SDO representations of a SOCO problem. We demonstrate that using arrow-head matrix transformation on the primal or dual SOCO problem, does not result in an arrow-head matrix variable on its dual. In fact, nothing forces the dual variable to be arrow-head. We usually end up with dense matrices which can be represented by either rank-one or full rank mappings, based on the position of SOCO solutions in the second-order cone. We propose low-rank to full rank mappings which are admissible, meaning that they preserve feasibility and objective function value. First of all, these mappings are not unique. One can come up with different mappings that satisfies these conditions. In addition, the dense structure of mapped solutions gives us an intuition about why solving SOCO as a SOCO problem is more efficient than solving as an SDO problem. In the SDO representation, we have to deal with dense matrices when solving the SDO counterpart. Furthermore, we investigated the relationship between the optimal partitions of these problems. The optimal partition of SOCO is an index-based partition, while that of SDO is a subspace-based partition. We discussed how these partitions map to each other based on the eigenvalue analysis of mapped solutions.

% As mentioned, the dense structure of the mapped solution gives us an intuition on why SDO is more difficult to solve than SOCO. To ensure that this is a contributing factor, further investigation on the theoretical complexity of solving the SDO representation is required. Based on the observation that our proposed mapping gives the freedom to represent a solution in the interior of second-order cone via rank-one mapping instead of full rank, may be useful in this analysis. Moreover, 
The theoretical study of the relationship between SOCO and its SDO counterparts remains with unsolved questions. One interesting question to look into is to study how degeneracy and singularity degree properties are affected throughout mappings. 

\clearpage

\appendix
\section{Proof of Theorem \ref{th:1}} \label{appsec:1}
\begin{proof}
To prove that the presented mapping in Theorem \ref{th:1} is admissible, we need to show that it complies with the definition of admissible mapping, i.e. preservation of feasibility and the objective function value. \\
\textbf{Step 1.} First, we show that if $(\ybar,\sbar) \in \Fcal_{\Dcal^1_{SOCO}}$, then $(y,S) = (\ybar,\arw(\sbar)) \in \Fcal_{\Dcal^{D}_{SDO}}$. Since $S=\arw(\sbar)$ and $\sbar\in \Lcal^n$, we know that $S$ is positive semidefinite \cite{alizadeh2003second}. The only thing remains to prove is that $\sum_{i=1}^n y_i \Avec_i + S=\Cvec$. Since $(\ybar,\sbar)$ is feasible, i.e., $A^T\ybar+\sbar=c$, we have 
\begin{align*}
    \sum_{i=1}^n y_i \Avec_i + S &=
    \begin{bmatrix}
    \sum_{i=1}^n \ybar_i a_{i1} + \sbar_1 & \sum_{i=1}^n \ybar_i a_{i2} + \sbar_2 &   ...  & \sum_{i=1}^n \ybar_i a_{in} + \sbar_n \\
    \sum_{i=1}^n \ybar_i a_{i2} + \sbar_2 & \sum_{i=1}^n \ybar_i a_{i1} + \sbar_1 &        &                                   \\
    \vdots                            &                                   & \ddots &                                   \\
    \sum_{i=1}^n \ybar_i a_{in} + \sbar_n &                                   &        & \sum_{i=1}^n \ybar_i a_{i1} + \sbar_1 
    \end{bmatrix}\\&=\begin{bmatrix}
    c_1 & c_2 &   ...  & c_n \\
    c_2 & c_1 &        &                                   \\
    \vdots                            &                                   & \ddots &                                   \\
    c_n &                                   &        & c_1 
    \end{bmatrix}=\Cvec.
\end{align*}
\textbf{Step 2.}  We show that if $\xbar = (\xbar_1, ..., \xbar_n) \in \Fcal_{\Pcal^1_{SOCO}}$, then matrix $X$ as defined in Theorem \ref{th:1}, belongs to $\Fcal_{\Pcal^{D}_{SDO}}$. By construction, $X$ is positive semidefinite and we need to just show that $\trace{(\Avec_iX)}=b_i$. Based on the construction of $X$ and feasibility of $\xbar$, i.e., $A\xbar=b$, we have
\begin{align*}
    \text{Tr}(\Avec_iX) &= a_{i1} (\sum_{i=1}^n X_{ii}) + 2a_{i2} X_{12} + ... + 2 a_{in} X_{1n}\\
    &=a_{i1}\xbar_1+ a_{i2} \xbar_2 +... +  a_{in} \xbar_n=b_i.
\end{align*}
\textbf{Step 3.} It is obvious that $b^T \ybar = b^T y$. Similar to step 2, we have
\begin{align*}
    \text{Tr}{(\Cvec X)} &= c_{1} (\sum_{i=1}^n X_{ii}) + 2c_{2} X_{12} + ... + 2 c_{n} X_{1n}\\
    &=c_{1}\xbar_1+ c_{2} \xbar_2 +... +  c_{n} \xbar_n=c^T\xbar.
\end{align*}
\textbf{Step 4.} In this step,  we  show that if $X \in \Fcal_{\Pcal^{D}_{SDO}}$, then $\xbar \in \Fcal_{\Pcal^1_{SOCO}}$. Suppose that $X \in \Fcal_{\Pcal^{D}_{SDO}}$, then it is positive semidefinite. We need to show that $\xbar \in \Lcal^n$. To do so, we utilize that in a positive semidefinite matrix every principal submatrix, in particular every 2-by-2 submatrix is positive semidefinite \cite{horn2012matrix}. Thus,
\begin{align*}
    |X_{ij}| \leq \sqrt{X_{ii} X_{jj}} \qquad \qquad\text{ for all } i = 1,...,n \text{ and }  \text{ for all } j = 1,...,n
\end{align*}
Thus, we have
\begin{align}
    X_{12}^2 + X_{13}^2 + ... + X_{1n}^2 &\leq (X_{11}X_{22}) + (X_{11}X_{33}) + ... + (X_{11}X_{nn})\\
                                         &\leq X_{11} (X_{22}+ X_{33} + ... + X_{nn}),\\
    \sqrt{X_{12}^2 + X_{13}^2 + ... + X_{1n}^2} &\leq \sqrt{X_{11} (X_{22}+ X_{33} + ... + X_{nn})}\\
                                                &\leq \frac{X_{11} + X_{22}+ X_{33} + ... + X_{nn}}{2} \label{eq:arigeo},
\end{align}
where (\ref{eq:arigeo}) is derived using the arithmetic-geometric mean inequality.
The expression can be rewritten as
\begin{align*}
    \sum_{i=1}^n X_{ii} &\geq 2\sqrt{X_{12}^2 + X_{13}^2 + ... + X_{1n}^2}\\
                        & =   \sqrt{(2X_{12})^2 + (2X_{13})^2 + ... + (2X_{1n})^2},                     
\end{align*}
or simply,
\begin{align*}
    \sum_{i=1}^n X_{ii} &\geq \sqrt{(2X_{12})^2 + (2X_{13})^2 + ... + (2X_{1n})^2},\\
    \xbar_1 &\geq  \sqrt{(\xbar_2)^2 + ... + (\xbar_n)^2}=||\xbar_{2:n}||_2,
\end{align*}
which shows that $\xbar \in \Lcal^n$.
Now, we need to prove that $A\xbar = b$. Thus, we have
\begin{align*}
    A_i\xbar &= a_{i1} \xbar_1 + a_{i2} \xbar_2 + ... + a_{in} \xbar_n \qquad \text{ for all } i = 1,...,m,\\
    &= a_{i1} (\sum_{i=1}^n X_{ii}) + a_{i2} (2X_{12}) + ... +  a_{in} (2X_{1n})\\
    &= \text{Tr}(\Avec_iX)=b_i.
\end{align*}
\textbf{Step 5.} Finally, we need to show that if $(y,S) \in \Fcal_{\Dcal^{D}_{SDO}}$, then $(\ybar,\sbar) \in \Fcal_{\Dcal^1_{SOCO}}$. As $(y,S) = (\ybar,\arw(\sbar))$, according to \cite{alizadeh2003second}, we know that $S = \arw(\sbar) \succeq 0$, implies $\sbar \in \Lcal^n$. Then, the only property that remains to prove is that $\sum_{i=1}^m \ybar_i A_i + \sbar = c$. By feasibility of $(y,S)$, we know that $\sum_{i=1}^n y_i \Avec_i + S = \Cvec$, i.e.,
\begin{align*}
    \begin{bmatrix}
    \sum_{i=1}^n \ybar_i a_{i1} + \sbar_1 & \sum_{i=1}^n \ybar_i a_{i2} + \sbar_2 &   ...  & \sum_{i=1}^n \ybar_i a_{in} + \sbar_n \\
    \sum_{i=1}^n \ybar_i a_{i2} + \sbar_2 & \sum_{i=1}^n \ybar_i a_{i1} + \sbar_1 &        &                                   \\
    \vdots                            &                                   & \ddots &                                   \\
    \sum_{i=1}^n \ybar_i a_{in} + \sbar_n &                                   &        & \sum_{i=1}^n \ybar_i a_{i1} + \sbar_1 
    \end{bmatrix}=\begin{bmatrix}
    c_1 & c_2 &   ...  & c_n \\
    c_2 & c_1 &        &                                   \\
    \vdots                            &                                   & \ddots &                                   \\
    c_n &                                   &        & c_1 
    \end{bmatrix},
\end{align*}
where each element of the matrix is a constraint in (\ref{mo:socod}) which means $\sum_{i=1}^m \ybar_i A_i + \sbar = c$.

Considering all of the previous steps, we conclude that presented mapping in Theorem \ref{th:1} is an admissible mapping.
\end{proof} 
\clearpage

\section{Proof of Theorem~\ref{th:3}} \label{appsec:2}
\begin{proof}
A rank-$n$ mapping of $\bar{x}$ can be constructed by Algorithm~\ref{alg:FRM}.
\begin{algorithm}[H]
\caption{Full Rank Mapping}\label{alg:FRM}
\begin{algorithmic}
\STATE $\pi^1 = \bar{x}$, 
\STATE Choose sufficiently small $\epsilon$
\FOR{$k=1:n-1$}
\STATE $\tau^k\gets \frac{\pi^k}{2}$ 
\STATE  
\begin{equation*}
    \begin{cases}
     \tau{}^k_{k+1} = \tau^k_{k+1}+\epsilon & \text{ if } \pi^k_{k+1}\geq 0\\
     \tau{}^k_{k+1} = \tau^k_{k+1}-\epsilon & \text{ if } \pi^k_{k+1}< 0
    \end{cases}
\end{equation*}
\STATE Calculate \begin{equation*}
    \beta^k =\frac{1}{2\sqrt{\frac{\tau{}^k_1 + \hat{\delta}^k}{2}}} (\tau{}^k_1 + \hat{\delta}^k, \tau{}^k_2, ..., \tau{}^k_n)^T,\text{ where }\hat{\delta}^k = \sqrt{(\tau{}^k_1)^2 - ||\tau{}^k_{2:n}||^2}.
    \end{equation*}
\STATE $\pi^{k+1}=\pi^k-\tau^k$
\ENDFOR
\STATE Calculate \begin{equation*}
    \beta^n =\frac{1}{2\sqrt{\frac{\pi{}^n_1 + \hat{\delta}^n}{2}}} (\pi{}^n_1 + \hat{\delta}^n, \pi{}^n_2, ..., \pi{}^n_n)^T,\text{ where }\hat{\delta}^n = \sqrt{(\pi{}^n_1)^2 - ||\pi{}^k_{2:n}||^2}.
    \end{equation*}
\end{algorithmic}
\end{algorithm}
If $\|\pi_{2:n}^k\|< \pi_1^k$ for all $k$, then we have
\begin{equation}\label{eq2}
\begin{aligned}
&0 < \frac{||\pi_{2:n}^k||}{2} \leq||\tau_{2:n}^k||<\tau_1^k \leq \frac{\pi_1^k}{2} < \pi_1^k,\\
    &\begin{cases}
    \frac{\pi_j^k}{2} \leq \tau_j^k & \text{ if } \pi_j\geq 0\\
    \frac{\pi_j^k}{2} \geq \tau_j^k & \text{ if } \pi_j< 0
    \end{cases}
     ,\qquad j = 2,...,n.
\end{aligned}
\end{equation}
We need to show that in all loops $\|\pi_{2:n}^k\|< \pi_1^k$. For $k=1$ this holds since $\|\xbar_{2:n}\|< \xbar_1$. Using induction, we need to prove $\|\pi_{2:n}^{k+1}\|<\pi_1^{k+1}$ assuming $\|\pi_{2:n}^{k}\|<\pi_1^{k}$. We have 
$$\pi_1^{k+1}-\|\pi_{2:n}^{k+1}\|=(\pi_1^{k}-\tau_1^k)-\|(\pi_j^{k}-\tau_j^k)_{2:n}\|\geq \tau_1^k-\|\tau_{2:n}^k\|>0.$$
Thus, $\pi_1^k>\|\pi_{2:n}^k\|$ for all $k$. The proposed $MR$ is also admissible since
\begin{align*}
    \sum_{j=1}^n MR_{jj}&= \sum_{k=1}^n\sum_{j=1}^n (\beta^k_j)^2= \xbar_1\\
    MR_{1j}&= \sum_{k=1}^n \beta^k_j\beta^k_1= \frac{\xbar_j}{2}.
\end{align*}
We need to prove that the generated $\beta^k$ vectors are linearly independent. We have
\begin{equation*}
    \beta^k =\frac{1}{2\sqrt{\frac{\tau{}^k_1 + \hat{\delta}^k}{2}}} (\tau{}^k_1 + \hat{\delta}, \tau{}^k_2, ..., \tau{}^k_n)^T,\text{ where }\hat{\delta}^k = \sqrt{(\tau{}^k_1)^2 - ||\tau{}^k_{2:n}||^2}.
    \end{equation*}
Let us define the matrix
\begin{align*}
    \mathbb{B}&=\begin{bmatrix}
\tau{}^1_1 + \hat{\delta}^1 & \tau{}^2_1 + \hat{\delta}^2& \dots&\tau{}^n_1 + \hat{\delta}^n\\
\tau{}^1_2 & \tau{}^2_2 & \dots&\tau{}^n_2 \\
\vdots & \vdots & \ddots&\vdots \\
\tau{}^1_n & \tau{}^2_n & \dots&\tau{}^n_n \\
\end{bmatrix}
\end{align*}

We still need to prove that $\mathbb{B}$ has full rank.
Now, w.l.o.g. we may assume that $\xbar\geq 0$. One can easily extend the following proof to general $\xbar$. Then we have
\begin{align*}
    \mathbb{B}&=\begin{bmatrix}
\tau{}^1_1 + \hat{\delta}^1 & \tau{}^2_1 + \hat{\delta}^2& \dots&\tau{}^n_1 + \hat{\delta}^n\\
\tau{}^1_2 & \tau{}^2_2 & \dots&\tau{}^n_2 \\
\vdots & \vdots & \ddots&\vdots \\
\tau{}^1_n & \tau{}^2_n & \dots&\tau{}^n_n \\
\end{bmatrix}\\
&=\begin{bmatrix}
\frac{\pi{}^1_1}{2} + \hat{\delta}^1 & \frac{\pi{}^2_1}{2} + \hat{\delta}^2& \dots&\pi{}^n_1 + \hat{\delta}^n\\
\frac{\pi{}^1_2}{2}+\epsilon & \frac{\pi{}^2_2}{2} & \dots&\pi{}^n_2 \\
\frac{\pi{}^1_3}{2} & \frac{\pi{}^2_3}{2}+\epsilon & \dots&\pi{}^n_3 \\
\vdots & \vdots & \ddots&\vdots \\
\frac{\pi{}^1_n}{2} & \frac{\pi{}^2_n}{2} & \dots&\pi{}^n_n \\
\end{bmatrix}\\
&=\begin{bmatrix}
\frac{\xbar_1}{2} + \hat{\delta}^1 & \frac{\xbar_1}{4} + \hat{\delta}^2& \dots&\frac{\xbar_1}{(2)^n} + \hat{\delta}^n\\
\frac{\xbar_2}{2}+\epsilon & \frac{\xbar_2}{4}-\frac{\epsilon}{2} & \dots&\frac{\xbar_1}{(2)^n}-\frac{\epsilon}{(2)^{n-1}} \\
\frac{\xbar_3}{2} & \frac{\xbar_3}{4}+\epsilon & \dots&\frac{\xbar_1}{(2)^n}-\frac{\epsilon}{(2)^{n-2}} \\
\vdots & \vdots & \ddots&\vdots \\
\frac{\xbar_n}{2} & \frac{\xbar_n}{4} & \dots&\frac{\xbar_1}{(2)^n}-\frac{\epsilon}{(2)} \\
\end{bmatrix}\\
&=\begin{bmatrix}
\frac{\xbar_1}{2}  & \frac{\xbar_1}{4} & \dots&\frac{\xbar_1}{(2)^n} \\
\frac{\xbar_2}{2} & \frac{\xbar_2}{4} & \dots&\frac{\xbar_1}{(2)^n} \\
\frac{\xbar_3}{2} & \frac{\xbar_3}{4} & \dots&\frac{\xbar_1}{(2)^n} \\
\vdots & \vdots & \ddots&\vdots \\
\frac{\xbar_n}{2} & \frac{\xbar_n}{4} & \dots&\frac{\xbar_1}{(2)^n}\\
\end{bmatrix}+\begin{bmatrix}
 \hat{\delta}^1 &  \hat{\delta}^2& \dots& \hat{\delta}^{n-1}& \hat{\delta}^n\\
\epsilon & -\frac{\epsilon}{2} & \dots&-\frac{\epsilon}{(2)^{n-2}}&-\frac{\epsilon}{(2)^{n-1}} \\
0 & \epsilon & \dots&-\frac{\epsilon}{(2)^{n-3}}&-\frac{\epsilon}{(2)^{n-2}} \\
\vdots & \vdots & \dots& \vdots &\vdots \\
0 & 0 & \dots&\epsilon &-\frac{\epsilon}{(2)} \\
\end{bmatrix},
\end{align*}
where
\begin{equation*}
    \begin{bmatrix}
    \hat{\delta}^1\\
    \hat{\delta}^2\\
    \vdots\\
    \hat{\delta}^n\\
    \end{bmatrix}=
    \begin{bmatrix}
    \sqrt{(\frac{\xbar_1}{2})^2-(\frac{\xbar_2}{2}+\epsilon)^2-(\frac{\xbar_3}{2})^2-\dots-(\frac{\xbar_n}{2})^2}\\
    \sqrt{(\frac{\xbar_1}{4})^2-(\frac{\xbar_2}{4}-\frac{\epsilon}{2})^2-(\frac{\xbar_3}{4}+\epsilon)^2-\dots-(\frac{\xbar_n}{4})^2}\\
    \vdots\\
    \sqrt{(\frac{\xbar_1}{(2)^n})^2-(\frac{\xbar_1}{(2)^n}-\frac{\epsilon}{(2)^{n-1}})^2-(\frac{\xbar_3}{4}-\frac{\epsilon}{(2)^{n-1}})^2-\dots-(\frac{\xbar_n}{(2)^n}-\frac{\epsilon}{2})^2}\\
    \end{bmatrix}.
\end{equation*}
By doing row eliminations on both matrices to sparsify the second matrix, we get
$$\begin{bmatrix}
\frac{\xbar_1}{2}  & 2\frac{\xbar_1}{4} & \dots&n\frac{\xbar_1}{(2)^n} \\
\frac{\xbar_2}{2} & 2\frac{\xbar_2}{4} & \dots&n\frac{\xbar_1}{(2)^n} \\
\frac{\xbar_3}{2} & 2\frac{\xbar_3}{4} & \dots&n\frac{\xbar_1}{(2)^n} \\
\vdots & \vdots & \ddots&\vdots \\
\frac{\xbar_n}{2} & 2\frac{\xbar_n}{4} & \dots&n\frac{\xbar_1}{(2)^n}\\
\end{bmatrix}+\begin{bmatrix}
 0& 0& \dots& 0& \bar{\delta}\\
\epsilon & 0 & \dots&0&0 \\
0 & \epsilon & \dots&0&0 \\
\vdots & \vdots & \dots&\vdots&\vdots \\
0 & 0 & \dots&\epsilon &0 \\
\end{bmatrix}$$
where
\begin{align*}
    \bar{\delta}= \frac{\hat{\delta}^1}{(2)^{n-1}}+\frac{\hat{\delta}^2}{(2)^{n-2}}+\dots+\frac{\hat{\delta}^{n-1}}{2}+\hat{\delta}^n.
\end{align*}
As we can see $\bar{\delta}$, as sum of strictly positive numbers, is strictly positive, and the matrix has full rank. By choosing small $\epsilon$, we can show that $\mathbb{B}$ has full rank. The last piece is to find appropriate value for $\epsilon$. We need to choose $\epsilon$ so that
\begin{align*}
    (\frac{\xbar_1}{2})^2-(\frac{\xbar_2}{2}+\epsilon)^2-(\frac{\xbar_3}{2})^2-\dots-(\frac{\xbar_n}{2})^2&>0\\
    (\frac{\xbar_1}{4})^2-(\frac{\xbar_2}{4}-\frac{\epsilon}{2})^2-(\frac{\xbar_3}{4}+\epsilon)^2-\dots-(\frac{\xbar_n}{4})^2&>0\\
    &\vdots\\
    (\frac{\xbar_1}{(2)^n})^2-(\frac{\xbar_1}{(2)^n}-\frac{\epsilon}{(2)^{n-1}})^2-(\frac{\xbar_3}{4}-\frac{\epsilon}{(2)^{n-1}})^2-\dots-(\frac{\xbar_n}{(2)^n}-\frac{\epsilon}{2})^2&>0.
    \end{align*}
Let $\varrho=(\xbar_1)^2-\sum_{i=2}^n(\xbar_i)^2>0$. Then, we have
\begin{align*}
    4\epsilon^2+4\xbar_2\epsilon-\varrho&<0 \\
    20\epsilon^2+(8\xbar_3-4\xbar_2)\epsilon-\varrho&< 0\\
    &\vdots
    \end{align*}
Since the second derivatives of all quadratic forms in these inequalities are positive, they are convex. Since coefficients of $\epsilon^2$ and $\rho$ are strictly positive, they have two distinct roots. Thus, any epsilon in the intersection of all these intervals satisfies the required conditions. The intersection of them is non-empty since we have a valid choice $\epsilon=\frac{\xbar_1 - \|\xbar_{2:n} \|}{2(n-1)}$ as discussed in Section~\ref{ssec:FRM}.
\end{proof}

\section{Proof of Theorem \ref{th:6}} \label{appsec:3}
\begin{proof}
Proof is similar to that of Theorem \ref{th:1}, showing that under the proposed setting $(\xbar, \ybar, \sbar)$ and $(X,y,S) = \Mcal(\xbar, \ybar, \sbar)$ satisfy the primal and dual constraints and preserves the objective function value. To this end, we need to take the following steps.

\textbf{Step 1.}
First, we show that if $\xbar \in \Fcal_{\Pcal^1_{SOCO}}$, then $X = \arw(\xbar) \in \Fcal_{\Pcal^{P}_{SDO}}$. First, we know that since $\xbar \in \Lcal^n$, then $X = \arw(\xbar) \succeq 0$. Next, we need to show that $X$ satisfies the SDO primal constraint. Thus,
\begin{align*}
    \trace{(\Avec_i X)} &= \frac{1}{n} a_{i1} (\sum_{i=1}^n X_{ii}) + \frac{2}{2} a_{i2} X_{12} + ... + \frac{2}{2} a_{in} X_{1n}\\
    &= a_{i1} \xbar_1 + a_{i2} \xbar_{2} + ... + a_{in} \xbar_n \\
    &= a_{(i)} \xbar = b_i, && \text{ for all } i = 1,...,m.
\end{align*}
The second equality represents the product of $i^{\text{\rm th}}$ row of matrix $A$ and vector $\xbar$. Hence, we can conclude that $X$ satisfies the SDO primal constraint. 

\textbf{Step 2.}
Next, we show that if $(\ybar,\sbar) \in \Fcal_{\Dcal^1_{SOCO}}$, then $(y,S) \in \Fcal_{\Dcal^{P}_{SDO}}$. By definition and using $u$ and $w$ in construction of $S$, we have $S \succeq 0$. Next, we need to show that $S$ satisfies the constraint of \eqref{mo:2dsdod}. Thus, since $(\ybar,\sbar)$ is feasible for \eqref{mo:socod}, i.e. $\sum_{i=1}^m \ybar_i a_{ij} + \sbar_j = c_j$ for all $j=1,...,n$,  we can write

\small
\begin{align*}
    \sum_{i=1}^m y_i \Avec_i + \sum_{h\neq 1, h < l} w_{hl} \Atilde_{hl} + \sum_{k=2}^n u_k \Ahat_k + S & = 
    \begin{bmatrix}
    \frac{1}{n} \sum_{i=1}^m y_i a_{i1} & \frac{1}{2} \sum_{i=1}^m y_i a_{i2} &   ...  & \frac{1}{2} \sum_{i=1}^m y_i a_{in} \\
    \frac{1}{2} \sum_{i=1}^m y_i a_{i2} & \frac{1}{n} \sum_{i=1}^m y_i a_{i1} &        &                                     \\
    \vdots                             &                                     & \ddots &                                      \\ \frac{1}{2} \sum_{i=1}^m y_i a_{in} &                                     &        & \frac{1}{n} \sum_{i=1}^m y_i a_{i1} 
    \end{bmatrix}\\
    & + 
    \begin{bmatrix}
    0        &  0       &  \hdots  & \hdots     & 0          \\
    0        &  0       &  w_{23}  & \hdots     & w_{2n}     \\
    \vdots   &  w_{23}  &  \ddots  & \ddots     & \vdots     \\
    \vdots   &  \vdots  &  \ddots  & \ddots     & w_{(n-1)n} \\ 
    0        &  w_{2n}  &  \hdots  & w_{(n-1)n} & 0          \\
    \end{bmatrix}\\
    & + 
    \begin{bmatrix}
    \sum_{k=2}^n u_k &    0      & \hdots & \hdots & \hdots &    0  \\
    0                &    -u_2   &    0   & \hdots & \hdots &    0  \\
    \vdots           &    0      & \ddots & \ddots &        & \vdots\\
    \vdots           & \vdots    & \ddots &   -u_i   & \ddots & \vdots\\
    \vdots           & \vdots    &        & \ddots & \ddots &    0  \\
    0                &    0      & \hdots & \hdots & 0      &    -u_n  \\
    \end{bmatrix}\\
    & + 
    \begin{bmatrix}
    \frac{\sbar_1}{n} - \sum_{k=2}^n u_k      &  \frac{\sbar_2}{2}     &   \hdots                    & \hdots                  & \frac{\sbar_n}{2}\\
    \frac{\sbar_2}{2}                       &  \frac{\sbar_1}{n}+u_2  &  -w_{23}           & \hdots                  &  -w_{2n} \\
    \vdots                                        &  -w_{23}           &  \frac{\sbar_1}{n}+u_3  & \ddots                  &  \vdots   \\
    \vdots                                        &   \vdots                    &   \ddots                    & \ddots                  &  -w_{(n-1)n} \\ 
    \frac{\sbar_n}{2}                        &  -w_{2n}           &   \hdots                    & -w_{(n-1)n} & \frac{\sbar_1}{n}+u_n \\
    \end{bmatrix}\\
    & = 
    \begin{bmatrix}
    \frac{1}{n} c_1    & \frac{1}{2}c_2 &  ...   & \frac{1}{2}c_n \\
    \frac{1}{2} c_2    & \frac{1}{n}c_1 &        &     \\
    \vdots &     & \ddots &    \\
    \frac{1}{2} c_n    &     &        & \frac{1}{n} c_1 
    \end{bmatrix} = \Cvec. 
\end{align*}
\normalsize
\textbf{Step 3.}
It is obvious that $b^T \ybar = b^T y$. Similar to step 1, we have
\begin{align*}
    \trace{(\Cvec X)} &= \frac{1}{n} c_1 (\sum_{i=1}^n X_{ii}) + \frac{2}{2} c_2 X_{12} + ... + \frac{2}{2} c_n X_{1n}\\
    &= c_1 \xbar_1 + c_2 \xbar_{2} + ... + c_n \xbar_n \\
    &= c^T \xbar.
\end{align*}
\textbf{Step 4.} In this step, we need to show that if $X \in \Fcal_{\Pcal^{P}_{SDO}}$, then $\xbar \in \Fcal_{\Pcal^1_{SOCO}}$. To this end, we know that as $X = \arw(\xbar)$ is positive semidefinite, then $\xbar \in \Lcal^n$. Next, we need to show that $A\xbar = b$. Thus, we have
\begin{align*}
    A_i \xbar &= a_{i1} \xbar_1 + a_{i2} \xbar_2 + ... + a_{in} \xbar_n \\
    & = \frac{1}{n} a_{i1} (n\xbar_1) + 2(\frac{1}{2} a_{i2}) \xbar_2 + ... + 2(\frac{1}{2} a_{in}) \xbar_n\\
    & = \trace{(\Avec_i X)} = b_i, && \text{ for all } i = 1,...,m.
\end{align*}
\textbf{Step 5.} Finally, we need to show that if $(y,S) \in \Fcal_{\Dcal^{P}_{SDO}}$, then $(\ybar,\sbar) \in \Fcal_{\Dcal^1_{SOCO}}$. Suppose that $S \in \Fcal_{\Dcal^{P}_{SDO}}$, then it is positive semidefinite. We need to show that $\sbar \in \Lcal^n$. To do so, recall that in a positive semidefinite matrix every principal submatrix, in particular every 2-by-2 submatrix is positive semidefinite \cite{horn2012matrix}. Thus,
\begin{align*}
    |S_{ij}| \leq \sqrt{S_{ii} S_{jj}} \qquad \qquad\text{ for all } i = 1,...,n \text{ and }  \text{ for all } j = 1,...,n
\end{align*}
Thus, we have
\begin{align}
    S_{12}^2 + S_{13}^2 + ... + S_{1n}^2 &\leq (S_{11}S_{22}) + (S_{11}S_{33}) + ... + (S_{11}S_{nn})\\
                                         &\leq S_{11} (S_{22}+ S_{33} + ... + S_{nn}),\\
    \sqrt{S_{12}^2 + S_{13}^2 + ... + S_{1n}^2} &\leq \sqrt{S_{11} (S_{22}+ S_{33} + ... + S_{nn})}\\
                                                &\leq \frac{S_{11} + S_{22}+ S_{33} + ... + S_{nn}}{2} \label{eq:arigeo2},
\end{align}
where (\ref{eq:arigeo2}) is derived using the arithmetic-geometric mean inequality.
The expression can be rewritten as
\begin{align*}
    \sum_{i=1}^n S_{ii} &\geq 2\sqrt{S_{12}^2 + S_{13}^2 + ... + S_{1n}^2}\\
                        & =   \sqrt{(2S_{12})^2 + (2S_{13})^2 + ... + (2S_{1n})^2},                     
\end{align*}
or simply,
\begin{align*}
    \sbar_1 &\geq  \sqrt{(\sbar_2)^2 + ... + (\sbar_n)^2}=||\sbar_{2:n}||_2,
\end{align*}
which according to definition of \eqref{eq:derivedS} shows that $\sbar \in \Lcal^n$. Next, we need to show that vector 
\begin{align*}
    \sbar = \begin{bmatrix}
    c_1 - \sum_{i=1}^{m} y_i a_{i1}\\
    \vdots \\
    c_n - \sum_{i=1}^{m} y_i a_{in},
    \end{bmatrix}
\end{align*}
satisfies the SOCO dual constraint. This is obvious since each entry of this vector is equal to the corresponding entry of vector $\sbar$ in the definition of \eqref{mo:socod}.
\end{proof}

\vskip 6mm
\noindent{\bf Acknowledgements}

\noindent
This research is supported by the National Science Foundation (NSF) under Grant No. 2128527.
\bibliographystyle{plain}
\bibliography{reference}
\end{document}